\documentclass{article}
\usepackage{amsmath, amsthm, amsfonts,tikz}
\author{}
\title{The Extended Persistent Homology Transform of manifolds with boundary}
\usepackage{xcolor}

\usetikzlibrary{arrows.meta,
                chains,
                positioning,
                shapes.geometric
                }
\usepackage{makecell}
\usepackage{sidecap}
\usepackage{caption}
\usepackage{subcaption}

\tikzset{
  treenode/.style = {shape=rectangle, rounded corners,
                     draw, align=center,},
  root/.style     = {treenode, font=\small}, 
  env/.style      = {treenode, font=\ttfamily\normalsize},
  dummy/.style    = {circle,draw}
}
\tikzstyle{level 1}=[level distance=2.8cm, sibling distance=10cm]
\tikzstyle{level 2}=[level distance=2.8cm, sibling distance=5cm]

\author{Katharine Turner$^1$, Vanessa Robins$^1$ and James Morgan$^2$}
\date{\small
$^1$Australian National University\\
$^2$University of Sydney
}
\usepackage{hyperref}

\theoremstyle{plain}
\newtheorem*{theorem*}{Theorem}
\newtheorem{theorem}{Theorem}[section]
\newtheorem{lemma}[theorem]{Lemma}
\newtheorem{cor}[theorem]{Corollary}
\newtheorem{corollary}[theorem]{Corollary}

\theoremstyle{definition}
\newtheorem{definition}[theorem]{Definition}

\newtheorem{example}[theorem]{Example}
\newtheorem{remark}[theorem]{Remark}
\newtheorem {proposition}[theorem] {Proposition}

\newcommand{\R}{\mathbb{R}}

\newcommand{\Crit}{\operatorname{Crit}}

\newcommand{\CS}{\text{CS}}

\newcommand{\PHT}{\operatorname{PHT}}
\newcommand{\ECT}{\operatorname{ECT}}
\newcommand{\sgn}{\operatorname{sgn}}
\newcommand{\Int}{\operatorname{int}}
\newcommand{\cl}{\operatorname{cl}}

\newcommand{\PH}{\operatorname{PH}}
\newcommand{\Ord}{\operatorname{Ord}}
\newcommand{\XPH}{\operatorname{XPH}}
\newcommand{\Rel}{\operatorname{Rel}}
\newcommand{\Ess}{\operatorname{Ess}}
\newcommand{\XPHT}{\operatorname{XPHT}}
\newcommand{\bt}{\mathfrak{b}}
\newcommand{\dt}{\mathfrak{d}}

\DeclareMathOperator{\im}{im}

\newcommand{\rel}{\operatorname{rel}}
\newcommand{\ord}{\operatorname{ord}}

\newcommand{\id}{\operatorname{id}}
\newcommand{\Eph}{\operatorname{Eph}}
\newcommand{\mf}{\mathfrak}
\renewcommand{\P}{\mathcal{P}}
\newcommand{\dist}{\operatorname{dist}}
\renewcommand{\d}{\mbox{dist}}
\newcommand{\PM}{\operatorname{PM}}
\usepackage{mathrsfs}

\begin{document}

\maketitle

\begin{abstract}
The Extended Persistent Homology Transform (XPHT) is a topological transform which takes as input a shape embedded in Euclidean space, and to each unit vector assigns the extended persistence module of the height function over that shape with respect to that direction. We can define a distance between two shapes by integrating over the sphere the distance between their respective extended persistence modules. By using extended persistence we get finite distances between shapes even when they have different Betti numbers. We use Morse theory to show that the extended persistence of a height function over a manifold with boundary can be deduced from the extended persistence for that height function restricted to the boundary, alongside labels on the critical points as positive or negative critical. We study the application of the XPHT to binary images; outlining an algorithm for efficient calculation of the XPHT exploiting relationships between the PHT of the boundary curves to the extended persistence of the foreground.
\end{abstract}

\section{Introduction}
\label{sec:introduction}

The fundamental goal in statistical shape analysis is to define and compute meaningful distances between different subsets of Euclidean space. A recent landmark-free approach to quantify both the geometry and topology of a shape is to use a topological transform such as the Persistent Homology Transform (PHT) or the Euler Characteristic Transform (ECT). Both of these transforms take a shape $M$, viewed as a subset $\R^n$, and associate to each direction $v\in S^{n-1}$ a shape summary obtained by scanning $M$ in the direction $v$, calculating the persistent homology ($\PH(M,v)$) and the Euler curve respectively. 

Different formulations of the $\PHT$ and $\ECT$ have been demonstrably useful in diverse applications including prediction of disease progression from the shapes of tumours (\cite{crawford2020predicting, shboul2019feature}), identification of different cultivars from the shapes of leaves \cite{zhang2021mfcis}, quantification of morphological variation of barley seeds \cite{amezquita2022measuring}, and identification of  structural differences among proteins \cite{tang2022topological}. This paper introduces an improved variant of this topological transform called the Extended Persistent Homology Transform (XPHT) and establishes properties that significantly reduce the time required to compute it.

A limitation of the $\PHT$ is it does not work well with shapes that have different Betti numbers (the ranks of the homology groups). For $M_1, M_2\subset \R^n$, the ($p$-)distance between their persistent homology transforms is defined as
$$\dist_p(\PHT(M_1), \PHT(M_2))^p=\int_{S^{n-1}} W_p(\PH(M_1, v), \PH(M_2, v))^p\,dv$$ 
where $W_p(\cdot, \cdot)$ is the p-Wasserstein distance.
If $M_1$ and $M_2$ have different Betti numbers, then $W_p(\PH(M_1,v),\PH(M_2,v))=\infty$, for all $v$, and thus 
$$\dist_p(\PHT(M_1),\PHT(M_2))=\infty.$$ One potential work-around would be to replace the Wasserstein distance with a different metric on the space of persistence modules, one where having different Betti numbers does not enforce infinite distance. A more satisfying approach is to replace persistent homology with extended persistent homology.

The theory of extended persistence for functions over a manifold $X$ was developed in \cite{cohen2009extending} to quantify the support of the essential homology classes of $X$ (these essential classes are the elements of $H_*(X)$).  Even when the domains have different Betti numbers we still have a finite Wasserstein distance between their extended persistence modules. This motivates the Extended Persistent Homology Transform ($\XPHT$) as a topological transform, which is defined in exactly the same manner as the PHT but replacing regular persistent homology with extended persistent homology. 
By quantifying the size of essential classes it is possible for $\XPHT$ to be stable with respect to the addition to, or removal of, ``small'' essential classes in the different domains. 
For example, if we add an isolated noisy pixel to a binary image then the change in the XPHT will be commensurate with the size of a pixel. This extra stability can provide greater power and robustness to statistical methods that use distances between shapes derived from the XPHT. 
As this paper is focused on computational aspects of the XPHT, comprehensive stability results are left as a future research direction.

We believe that extended persistence is currently under-utilised within applied topology and this paper addresses three potential obstacles. Firstly, we make extended persistence modules more theoretically accessible by placing them within a generalised framework that includes both regular persistence as well as extended persistence. Secondly, we provide motivation with an important example (in the form of the $\XPHT$) where using extended persistence provides a qualitative improvement in usefulness. Lastly, we provide insights on how to ease the computation of extended persistence in the important case of height functions, with implemented code for binary images.

\subsection{Outline of paper} 

The mathematical treatment of the XPHT and algorithms to compute it requires the adaption and extension of many standard definitions within applied topology.  We cover this material in some detail to make the paper more self-contained and to provide a cohesive perspective on results from different areas of the literature.

The original definition of extended persistence in~\cite{cohen2009extending} is made for functions defined on a smooth or piecewise-linear (PL) manifold and concatenates two homology sequences, the standard inclusion-induced persistent homology sequence for the sublevel set filtration, followed by a descending relative homology sequence for superlevel sets.  
In section~\ref{sec:extended}, we reformulate this as a persistence module over a totally ordered set, with all transition maps defined as those induced on relative homology by inclusions of a pair of spaces.  These spaces are defined by a real-valued function on a triangulated manifold with boundary, $f: M \to \R$.   
We then establish a relationship between the intervals of extended persistence modules of $f$ and $(-f)$, which is one of the results required to reduce computation time for the XPHT. 
 
In section~\ref{sec:wasserstein} we generalise the definition of Wasserstein and bottleneck distances between persistence diagrams to apply to persistence modules over a totally ordered metric space, with a defined set of ephemeral (zero-length) intervals. The Wasserstein and bottleneck distances are optimal transport metrics with transport plans that include a bijection between chosen subsets of intervals and then subsets of unmatched intervals. To define the cost of a transportation plan we need a distance between intervals and cost of having an interval unmatched.
We show our definition agrees with the existing definitions of bottleneck distance between extended persistence diagrams.

A key theoretical insight of our work, and one which makes the $\XPHT$ feasible to compute, is that for manifolds with boundary embedded in $\R^n$ the extended persistent homology of a height function over $M$ can be deduced from the persistent homology of the same height function restricted to $\partial M$.  This is the topic of section~\ref{sec:morse}. 
The proof of this insight requires ideas from Morse theory for manifolds with boundary, in both the smooth and piecewise-linear settings.  This background material is covered in section~\ref{sec:MorseSetup}. 
We also precisely state the relationship between birth and death parameters of  extended persistence in terms of the different kinds of critical points of a smooth or PL Morse function on a  manifold with boundary.
Section~\ref{subsec:XPH_boundary} then develops results specifically for the case of a directional height function. 
It is worth noting that any subset of $\R^n$ with positive weak feature size is arbitrarily close to a $n$-manifold with boundary by taking an $\epsilon$ expansion. This means the restriction to $n$-manifolds with boundary is reasonable from an application standpoint.
 
Adapting the definition of the persistent homology transform (PHT) to extended persistence is straightforward.  We cover this material in section~\ref{sec:xpht}.

Shape analysis of objects in digital images is an application domain with wide interest. 
Objects in binary images can be modelled as two dimensional manifolds with boundary lying in the plane, so our XPHT results  apply. 
In section~\ref{sec:Images} we define boundary curves that separate foreground and background connected components consistent with a chosen digital adjacency, and show that these boundary curves are disjoint simple closed PL 1-manifolds.  
Digital grids create degeneracy in the height function critical values, so we derive additional results that establish the correctness of our implemented algorithms.
Finally, in Section~\ref{sec:Implementation} we illustrate our R-package implementation by comparing the XPHT of the letters `A' and `g' rendered in a variety of standard fonts.  We find the $\XPHT$ of the upper case `A' naturally  separates  the serif and sans-serif fonts, and that the $\XPHT$ of the lower case `g' naturally separates the single-storey and the double-storey fonts. 

 \subsection{Relation to the Alexander Duality for Extended Persistence } 
A form of Alexander Duality for extended persistence was proved in \cite{edelsbrunner_alexander_2012}. 
That paper considers the decomposition of the sphere into two sets $U$, $V$ with $U\cup V=S^{n}$ and $U\cap V$ a $(n-1)$-manifold, and proves results about the extended persistence of a perfect Morse function $f$  over these sets. 
A perfect Morse function over $S^{n}$ is a smooth function with exactly two critical points, one minimum and one maximum. 
Edelsbrunner and Kerber prove that the extended persistence module of $U\cap V$ is the direct sum of those for $U$ and $V$ (with minor adjustments for homology dimension zero). The statement of our Theorem \ref{thm:oplus} is effectively a special case of their result. 
However, our proof is very different as it is based on Morse theory instead of Alexander Duality. 
Another key difference in our results is that we show how the extended persistence module for $U\cap V$ splits into the two different parts (Theorem \ref{thm:submodule}); this is not established in~\cite{edelsbrunner_alexander_2012}. 
Since our ultimate goal is to calculate the extended persistence of $U$ from that of $U\cap V$ this splitting criteria is pivotal.

\section{Extended persistence modules}
\label{sec:extended} 
\subsection{Persistence modules over totally ordered sets}

Commonly persistence modules are defined with an underlying parameter space a subset of $\R$ but they can be defined where the parameter space is a totally ordered set. This approach will make working with extended persistence substantially cleaner and more intuitive as we will want to split our parameter space into ordinary and relative homology parameter types.  

\begin{definition}
A \emph{totally ordered set} $(\Theta,\leq)$ is a set $\Theta$ with a relation $\leq$ which is
\begin{itemize}
\item Reflexive: that is $\alpha \leq \alpha$ for all $\alpha \in \Theta$,
\item Antisymmetric: that is $\alpha \leq \beta$ and $\beta \leq \alpha$ implies $\alpha=\beta$, 
\item Transitive: that is $\alpha\leq \beta$ and $\beta\leq \gamma$ implies $\alpha\leq \gamma$, and
\item Comparable: for all $\alpha,\beta$ either $\alpha\leq \beta$ or $\beta\leq \alpha$.
\end{itemize}
\end{definition}

\begin{definition}
Fix a field $\mathbb{F}$ and $\Theta$ a totally ordered set. A \emph{persistence module} $\mathcal{P}$ over $\Theta$ is a family $\{V_\alpha\}_{\alpha\in \Theta}$ of $\mathbb{F}$-vector spaces indexed by elements of $\Theta$, together with a family of homomorphism $\{\varphi_\alpha^\beta:V_\alpha\to V_\beta\}$ such that  $\varphi^\gamma_\alpha= \varphi^\gamma_\beta \circ \varphi^\beta_\alpha $ for all $\alpha \leq \beta \leq \gamma$, and $\varphi_\alpha^\alpha= \id V_\alpha$. We call the $\varphi_\alpha^\beta$ the \emph{transition maps}. We say $\mathcal{P}$ is \emph{pointwise finite dimensional} if the $V_\alpha$ are finite dimensional for all $\alpha\in \Theta$. 
\end{definition}

In the algebraic theory of persistence modules there are often technical requirements about tameness, and being pointwise finite dimensional will generally be a sufficient condition. This is a very reasonable assumption in almost any application. The most important algebraic result is the decomposition theorem. This gives a complete yet discrete description of a persistence module up to isomorphism. We will decompose persistence modules into sums of interval modules, but first we must define interval persistence modules.

We are all familiar with intervals that are subsets of the real line. 
We  generalise this notion to any totally ordered set as follows.

\begin{definition}
An \emph{interval} in a totally ordered space $(\Theta, \leq)$ is a subset $I\subset \Theta$ such that for all $\alpha\in \Theta$ either $\alpha \in I$, or $\alpha \leq \theta$ for all $\theta\in  I$,  or $\theta\leq \alpha$ for all $\theta\in  I$. 
An \emph{interval module} over an interval $I$ is a persistence module $\mathcal{I}_I$ with attached vector spaces 
\begin{align*}
V_\alpha&=\begin{cases}
\mathbb{F} \text{ for }\alpha \in I\\
0 \text{ for } \alpha \notin I
\end{cases}
\end{align*}
and transition maps are the identity, $\operatorname{id}_\mathbb{F}$, when both domain and codomain are $\mathbb{F}$ and $0$ otherwise.

For each interval module $\mathcal{I}_I$ we call $\mf{b}(\mathcal{I}_I)=\inf{I}$ the \emph{birth} parameter and $\mf{d}(\mathcal{I}_I)=\sup{I}$ the \emph{death} parameter.
\end{definition}

The nomenclature of ``interval'' was introduced for persistence modules with parameter space $\R$ but it is still reasonable even in the generalised setting of totally ordered sets. 
If we can map the totally ordered set to a subset of the real line, say $f:\Theta \to \R$, in a way that respects the order relation, then we can view each interval module as having support $f^{-1}(I)$ where $I\subset \R$ is some interval.
 
\begin{theorem}[\cite{crawley2015decomposition} Theorem 1.1]\label{thm:intervaldecomposition}
A pointwise finite dimensional persistence module over any subset of $\mathbb{R}$ admits an \emph{interval decomposition}. That is, there is a multiset of intervals $S$  such that the module is isomorphic to a direct sum of interval modules
$$\bigoplus\limits_{I\in S} \mathcal{I}_I$$
where each $\mathcal{I}_I$ is an interval module. This decomposition is unique up to isomorphism. 
\end{theorem}

For the rest of the paper we will be assuming all persistence modules are pointwise finite dimensional and that the the underlying parameter space is equivalent to a subset of $\mathbb{R}$ (with respect to the order relation), and thus we can always assume an interval decomposition occurs.

Given a persistence module $\mathcal{P}=\bigoplus\limits_{I\in S^{\mathcal{P}}} \mathcal{I}_I$ we will use 
$$\mf{b}(\mathcal{P})=\{\mf{b}(\mathcal{\mathcal{I}_I}):I\in S^{\mathcal{P}}\}\qquad \text{ and } \qquad 
\mf{d}(\mathcal{P})=\{\mf{d}(\mathcal{\mathcal{I}_I}):I\in S^{\mathcal{P}}\}$$
to denote the multiset of birth parameters and death parameters in the interval decomposition of $\mathcal{P}$. 

Readers may be familiar with persistence diagrams. Persistence diagrams are a graphical representation of a persistence barcode. If we take our ordered set to be $\R$ then the parameters are real numbers. We can represent each interval module in the persistence module decomposition by a point in $\mathbb{R}^2$ with the x-coordinate the birth parameter, and its y-coordinate the death parameter. We then construct the \emph{persistence diagram} as the resulting multiset of points in $\R^2$ together with all the points along the diagonal in the plane.

\subsection{Extended persistence}
\label{ssec:ext_pers}

Extended persistence combines the regular filtration of sublevel sets for $f: M \to \R$ with a filtration of \emph{relative} homology groups of $M$ relative to superlevel sets of $f$. This provides a wealth of extra information about the structure of $M$, especially in the case that $M$ is a manifold with boundary.

We first recall the definition of relative homology, and the maps induced by the inclusion of a pair.
Given a subcomplex $X\subset Y$ we observe that the boundary map on $C_*(Y)$ leaves $C_*(X)$ invariant. This means we can define a chain complex $C_*(Y,X)$ where $C_k(Y,X)=C_k(Y)/C_k(X)$ and the boundary map is
$$\partial_k^{(Y,X)}(\alpha +C_k(X))=\partial(\alpha)+C_{k-1}(X).$$
We can then define the relative homology groups by 
$$H_k(Y,X)=\ker \partial_k^{(Y,X)}/\im \partial^{(Y,X)}_{k+1}.$$
Relative homology is a generalisation of normal homology as $H_k(Y)=H_k(Y,\emptyset)$.

If $Y\subset B$ and $X\subset A\subset B$   we have an inclusion of pairs $(Y,X)\subset (B,A)$. This inclusion of pairs induces a map between their relative homology groups, $\iota_*: H_k(Y,X) \to H_k(B,A)$, with $\iota(\alpha +C_k(X))=\alpha+C_k(A).$

We are now ready to define the extended persistence module as a form of persistence modules. The parameter space over which the persistence module is constructed will be the union of two sets --- one corresponding to ordinary homology and the other corresponding to relative homology. Set $O= \{(t,\text{Ord}): t\in \R\} $ and $R= \{(t,\text{Rel}): t\in \R\}$.
Let $\Theta=O \cup R$. We define a total order over $\Theta$ by 
\begin{align*}
(s,\text{Ord})&<(t,\text{Ord}) \text{ when }s<t\\
(s,\text{Rel})&<(t,\text{Rel})  \text{ when }s>t\\
(s,\text{Ord})&<(t,\text{Rel}) \text{ for all }s,t
\end{align*}

We then assign vector spaces to each $\theta \in \Theta$  defined in terms of sublevel and superlevel sets. 
As input we have a topological space $M$ with a bounded function $f:M \to \R$.
Let $M_s=f^{-1}(-\infty,s]$ and $M^s=f^{-1}[s,\infty)$ denote the sublevel and superlevel sets of $f:M\to\R$. We assign the vector spaces as $V_{(t,\text{Ord})}=H_k(M_t, \emptyset)$ and $V_{(t,\text{Rel})}=H_k(M,M^t)$.
The transition maps are the natural ones induced by inclusions of a pair.

The compositions of induced maps of inclusions of a pair is the corresponding induced map by inclusion. This means that the transition maps commute as needed and we have constructed a persistence module.

Each interval in the interval decomposition will be supported over some interval of $\Theta$ which will be one of three types; if the supports contains only the parameters in $O$ we call it \emph{ordinary}, if the support only contains parameters in $R$ we call it \emph{relative}. Finally, the persistent homology class might exist for parameters spanning both $O$ and $R$, in which case we call it \emph{essential}. Essential persistent homology classes exist in the vector space $H_k(M,\emptyset)= H_k(M)$ and in classical persistent homology are assigned a death parameter of infinity.  
The object in Fig.~\ref{fig:snail} illustrates the parameter space $\Theta$ and has one class of each type.

\begin{figure}
\includegraphics[width=\textwidth]{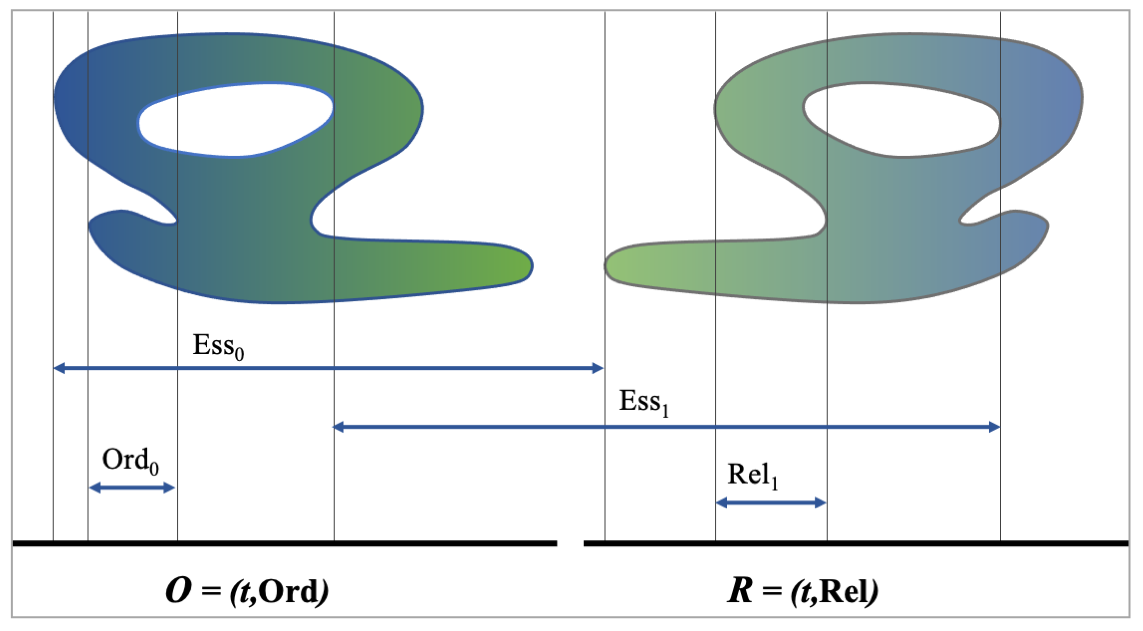}
\caption{An illustration of extended persistence intervals for a rather abstract snail, $M$. The function $f: M \to \R$ is simply the $x$-coordinate and the function value is denoted by the blue-green colour gradient.  We have drawn a copy of $M$ with its $x$-coordinate reflected to illustrate the superlevel sets used in the relative part of the sequence. }
\label{fig:snail}
\end{figure}

\begin{remark}
To preempt any confusion, we note a difference in our nomenclature from some papers, including \cite{cohen2009extending}. 
What we call essential classes above are instead called ``extended''.  We prefer the term ``essential'' as these classes do indeed correspond to the essential classes of $M$. Furthermore it means we can use ``extended'' to refer to any class in the extended persistence module.
\end{remark}

We can partition the elements of the interval decomposition of extended persistent homology into three sets depending on whether they are ordinary, relative or essential.  
Following \cite{carlsson2019parametrized} we can further split the essential classes into \emph{positive} and \emph{negative} types.  For an essential class with birth time $(s,\text{Ord})$ and death time $(t, \text{Rel})$, we say it is positive if $s<t$ and negative if $s>t$. 

We can express the extended persistence module as a direct sum of ordinary, relative and essential persistence modules.  For an extended persistence module constructed from sublevel and superlevel set filtrations of $f:M \to \R$ denote these submodules by $\Ord_k(M,f)$, $\Rel_k(M,f)$ and $\Ess^+_k(M,f)$ and $\Ess^-_k(M,f)$, which are each persistence modules over $\R$. 
For $\Rel_k(M,f)$ and  $\Ess^-_k(M,f)$ the order of parameters in $\R$ is reversed --- that is, the real value associated with the birth time is larger than the real value associated with the death time. 
In the case of subsets of $\R^2$ (cf.\ the example in Fig.~\ref{fig:snail}) we will show that $\Ess_0=\Ess_0^+$ and $\Ess_1=\Ess_1^-$ and thus we do not need to indicate the sign of the essential classes.

\subsubsection{Duality}  
 
There is a form of duality between the ordinary persistent homology of $f:M\to \R$ and the relative persistent homology of $(-f):M\to \R$. This follows from results in \cite{de2011dualities} but that paper uses substantially different notation to us. 
Furthermore, that paper considers filtrations of simplicial complexes, a context where we cannot naively switch between sublevel and superlevel sets.  For these reasons, we rewrite their proposition to suit the requirements of our setting. 

\begin{proposition}[Proposition 2.4 in \cite{de2011dualities}]\label{prop:duality}
Let $\mathbb{M}=\{M_t\}$ be a filtration of simplicial complexes.  Let $\PH_k(\mathbb{M})$ be the persistence module of $k$-dimensional persistent homology of the filtration $\mathbb{M}$.  Let $\PH^0_k(\mathbb{M})$ be the restriction of $\PH_k(\mathbb{M})$ to persistence classes with finite lifetimes. Let $\PH_{k+1}(M_\infty, \mathbb{M})$ be the persistence module of relative homology classes $H_{k+1}(M_\infty, M_t)$ and let $\PH_{k+1}^0(M_\infty, \mathbb{M})$ be the restriction of  $\PH_{k+1}(M_\infty, \mathbb{M})$ to persistence classes with finite lifetimes. 
Then  $\PH^0_k(\mathbb{M})$  and $\PH_{k+1}^0(M_\infty, \mathbb{M})$ are isomorphic.
\end{proposition}

\begin{corollary}\label{cor:duality}
Let $M$ be a finite simplicial complex, with vertex set $V$, and geometric realisation $|M|$. Let $f: |M|\to \R$ be a continuous map such that on each cell $f$ is the linear interpolation of the values on its vertices. We have a bijection $\rho$ between the interval modules in the interval decomposition of $\Ord_k(M,f)$ to that of $\Rel_{k+1}(M,(-f))$ with
$$\rho(\mathcal{I}_{[(b,\ord), (d, \ord))})=\mathcal{I}_{[(-b, \rel),(-d,\rel))}$$
 \end{corollary}

\begin{proof}
The $\PH_{k+1}^0(M_\infty, M_t)$ of  \cite{de2011dualities} is the relative homology of $M$ with respect to the (increasing $t$) sequence $M_t = f^{-1}\left(-\infty, t \right]$.  But $ f^{-1}\left(-\infty, t \right]  = (-f)^{-1}\left[-t, \infty \right) $, so the sequence $M_t$ of sublevel sets of $f$ is identical to a sequence of superlevel sets, $M^s$, of $(-f)$, with $s = -t$.  
Note that when the filtration is expressed as superlevel sets of $(-f)$, the parameter $s$ is a decreasing one, as used in the relative part of an extended persistence module. 

From Proposition \ref{prop:duality}, we have a bijection between the intervals in the interval decompositions with $\mathcal{I}_{[b,d)}\subset \PH_k^0(\mathbb{M})$ matched with  $\mathcal{I}_{[b,d)}\subset \PH_{k+1}^0(M,\mathbb{M})$. Composing this with the reparameterisation to superlevel set notation we have  $\mathcal{I}_{[b,d)}\subset \PH_{k+1}^0(M,\mathbb{M})$ rewritten as  $\mathcal{I}_{[(-b,\rel),(-d, \rel))}\subset \Rel(M, (-f)).$
\end{proof}

We note that this duality result is quite different from the duality theorem of~\cite{cohen2009extending}, which is proved in the case that $M$ is a triangulated $d$-manifold.  That paper goes on to also establish a symmetry theorem for extended persistence for functions over manifolds without boundary, which we discuss in our notation and context below. 

\subsubsection{Symmetry}

In the case that $M$ is a manifold we find that the information content in extended persistence modules is greatly reduced by the isomorphisms established in the following result. 

\begin{proposition}[Symmetry theorem of \cite{cohen2009extending}]\label{prop:symmetry}
Let $M$ be a triangulated $d$-manifold and $f: M \to \R$ be a piecewise-linear function interpolating the values on the vertices of $M$.  There are bijections, $\psi_{\bullet}$, between submodules of extended persistence for $f$ and $(-f)$ as follows:
\begin{align*}
    \psi_{O} &: \Ord_{k}(M,f) \to  \Ord_{d-k-1}(M,-f) &\mathcal{I}_{[(b,\ord), (d, \ord))} \mapsto \mathcal{I}_{[(-d, \ord),(-b,\ord))} \\
    \psi_{E} &: \Ess_{k}(M,f) \to  \Ess_{d-k}(M,-f) &\mathcal{I}_{[(b,\ord), (d, \rel))} \mapsto \mathcal{I}_{[(-d, \ord),(-b,\rel))} \\
    \psi_{R} &: \Rel_{k}(M,f) \to  \Rel_{d-k+1}(M,-f) &\mathcal{I}_{[(b,\rel), (d, \rel))} \mapsto \mathcal{I}_{[(-d, \rel),(-b,\rel))} 	
\end{align*}
\end{proposition}
\begin{remark}
We note that \cite{cohen2009extending} has a typographical error in the dimensions for the relative homology classes.
\end{remark} 
\begin{proof}
As in \cite{cohen2009extending}, first use Lefschetz duality $H_k( X, \partial X) \leftrightarrow H_{d-k}(X, \emptyset) $ with $X = M_t$ and the excision theorem to see that 
$$  H_k( M, M^t) = H_k( M_t, \partial M_t) = H_{d-k} (M_t, \emptyset). $$ 
Combined with the inclusion-induced maps on homology, this gives a bijection between the finite intervals of ordinary and relative homology in complementary dimensions:  $\Rel_{k}(M, f) \leftrightarrow \Ord_{d-k} (M,f)$, 
with $\mathcal{I}_{[(b,\rel), (d, \rel))} \mapsto \mathcal{I}_{[(d,\ord), (b, \ord))}$.
The same relationship holds for the essential homology classes: $\Ess_k(M,f) \leftrightarrow \Ess_{d-k}(M,f)$,  
with $\mathcal{I}_{[(b,\ord), (d, \rel))} \mapsto \mathcal{I}_{[(d, \ord), (b, \rel))}$.
Note these bijections are those established by the duality theorem of \cite{cohen2009extending}. 
Combined with the duality result~\ref{cor:duality} above, we now see that 
\begin{align*}
	\Rel_{k}(M, f) &\leftrightarrow \Ord_{d-k} (M,f) \leftrightarrow \Rel_{d-k+1} (M, -f) \\
	\Ess_k(M,f) &\leftrightarrow \Ess_{d-k}(M,f)  \leftrightarrow \Ess_{d-k}(M,-f) \\
	\Ord_{k}(M, f) &\leftrightarrow \Rel_{d-k} (M,f) \leftrightarrow \Ord_{d-k-1} (M, -f) 
\end{align*}
Composing the two bijections establishes the maps $\psi_{\bullet}$ in each case. 
\end{proof}

\begin{remark}
Our application to binary images has data $M$ that are manifolds with boundary, so the duality and symmetry theorems of \cite{cohen2009extending} do not apply directly. We will use the duality result of \cite{de2011dualities} to reduce the number of directions required when computing the extended persistent homology transform, since it gives a bijection between the intervals for height filtrations in opposite directions. Since the boundary $\partial M$ of a manifold with boundary $(M, \partial M)$ is a manifold we will be able to use the symmetry result to characterise the essential classes in $\Ess_0(\partial M,f)$ and $\Ess_{n-1}(\partial M,f)$.
\end{remark}

\section{Wasserstein distance between extended persistence modules}
\label{sec:wasserstein}  
\subsection{Wasserstein distances between persistence modules}

There are many possible metrics between persistence modules, and various representations of them. In this paper we restrict our attention to Wasserstein distances. 
Wasserstein distances between persistence modules are usually defined in terms of the points in their corresponding persistence diagrams. 
However, given our desire to study extended persistence, we rephrase the definitions here in terms of persistence modules over a totally ordered set. 
Wasserstein distances are a form of optimal transport metric. 
A transportation plan between two persistence modules matches subsets of intervals from each,   
with the remaining unmatched intervals paired with an ephemeral interval. 
Since every persistence module considered in this paper is isomorphic to a direct sum of interval modules it is sufficient to define our transportation plans between persistence modules written in this form.

\begin{definition}
Let $\Theta$ be a totally ordered set and $\P=\bigoplus_{I_i\in S^{\mathcal{P}}} \mathcal{I}_{I_i}$ and $\mathcal{Q}=\bigoplus_{I_j\in S^{\mathcal{Q}}}\mathcal{I}_{I_j}$ persistence modules over $\Theta$. A \emph{transportation plan} between $\P$ and $\mathcal{Q}$ is a triple $T=(\hat{S}^\mathcal{P},\hat{S}^\mathcal{Q},\rho)$ where $\hat{S}^\mathcal{P}\subset S^\mathcal{P}$, $\hat{S}^\mathcal{Q}\subset S^\mathcal{Q}$ and $\rho:\hat{S}^\mathcal{P} \to \hat{S}^\mathcal{Q}$ is a bijection. We call the intervals in $\hat{S}^\mathcal{P}$ and in $\hat{S}^\mathcal{Q}$ \emph{matched} intervals in $T$, and we call the intervals in $S^\mathcal{P}\backslash \hat{S}^\mathcal{P}$ and in $S^\mathcal{P}\backslash \hat{S}^\mathcal{Q}$ \emph{unmatched} intervals in $T$.
\end{definition}

 Each transportation plan has an associated cost, constructed analogously to an $L^p$ function metric.  
This in turn depends on the metric used to measure distance between points in $\Theta$, which we define below. 

\begin{definition}
We call $(\Theta, \leq, \d)$ a \emph{totally ordered metric space} if $(\Theta, \leq)$ is a totally ordered set, and $\dist$ is an extended metric over $\Theta$ such that $\d(\beta,\gamma)\leq \d (\alpha,\gamma)$ and $\d(\alpha,\beta)\leq \d(\alpha,\gamma)$ whenever $\alpha\leq \beta\leq \gamma$.  
\end{definition}

From the metric on $\Theta$ we obtain a $p$-distance between intervals over $\Theta$, analogous to the $l^p$ distance between points in $\R^2$. Given two intervals $\mathcal{I}$ and $\mathcal{I}'$, the $p$-distance (for $p\in [1, \infty)$) is defined as 
$$\dist_p(\mathcal{I},\mathcal{I}')= \big(\d(\mf{b}(\mathcal{I}),\mf{b}(\mathcal{I}'))^p+ \d(\mf{d}(\mathcal{I}),\mf{d}(\mathcal{I}'))^p \big) ^{1/p}.$$ 
The bottleneck, or $\infty$-distance, between intervals is $$\dist_\infty(\mathcal{I},\mathcal{I}')= \max\{\d(\mf{b}(\mathcal{I}),\mf{b}(\mathcal{I}')), \d(\mf{d}(\mathcal{I}),\mf{d}(\mathcal{I}'))\}.$$ 

Note that for general interval modules this is actually a pseudo-distance as it cannot distinguish between intervals with open or closed endpoints. 
However, if the persistence modules are constructed from filtrations involving closed sublevel and superlevel sets then the intervals are always half-open, including the birth parameter and not including the death parameter. When restricted to such half-open interval modules the above definition of $\dist_p$ will satisfy the identity of indiscernibles, making it an actual distance. Throughout this paper we will work exclusively with persistence modules that have these half-open intervals.

The final ingredient we need before defining the transportation plans and their costs is the notion of an ``empty interval''. 
For persistence diagrams these are points on the diagonal, corresponding to intervals of zero length in the usual setting of persistence modules over $\R$.  
In the general definition of Wasserstein distance we are allowed to fix any subset of interval modules to perform this role. 
We call this set the \emph{ephemeral intervals} denoted $\Eph$.  
This name is inspired by the definition of an ephemeral persistence module as one with distance zero to the trivial persistence module (see \cite{chazal2014observable}).

We now define the cost of a transportation plan using $p$-distances between intervals where the unmatched intervals of a plan are costed by their distance to the set of ephemeral intervals.  
\begin{definition}
Let $\Theta$ be a totally ordered set; $\mathcal{P}=\bigoplus_{a\in S^\mathcal{P}} \mathcal{I}_a$ and $\mathcal{Q}=\bigoplus_{b\in S^\mathcal{Q}}\mathcal{I}_b$ be persistence modules over the ordered metric space $(\Theta, \leq, \d)$. 
Let $\Eph$ denote the set ephemeral intervals over $\Theta$. 
Let $T=(\hat{S}^\mathcal{P}, \hat{S}^\mathcal{Q},\rho)$ be a transportation plan between $\mathcal{P}$ and $\mathcal{Q}$.
For $p\in [1,\infty)$ we define the $p$-cost of $T$ by 
\begin{align*}
c_p(T)^p =& \sum_{a\in \hat{S}^\mathcal{P}} \d_p(\mathcal{I}_a, \mathcal{I}_{\rho(a)})^p +\sum_{a\in S^\mathcal{P}\backslash \hat{S}^\mathcal{P}} \inf_{\mathcal{I}\in \Eph} \{\d_p(\mathcal{I}_a,\mathcal{I})^p\}\\
&  + \sum_{b\in S^\mathcal{Q}\backslash \hat{S}^\mathcal{Q}}  \inf_{\mathcal{I}\in \Eph} \{\d_p(\mathcal{I}_b,\mathcal{I})^p\}
\end{align*}
and 
\begin{align*}
c_\infty(T) = \max\Big\{ &\sup_{a\in\hat{S}^\mathcal{P}} \{\d_\infty(\mathcal{I}_a,\mathcal{I}_{\rho(a)})\},
\sup_{a\in S^\mathcal{P}\backslash \hat{S}^\mathcal{P}} \{ \inf_{\mathcal{I}\in \Eph} \{\d_\infty(\mathcal{I}_a,\mathcal{I})\}\}, \\
 &\sup_{b\in S^\mathcal{Q}\backslash \hat{S}^\mathcal{Q}} \{ \inf_{\mathcal{I} \in \Eph} \{\d_\infty(\mathcal{I}_b,\mathcal{I})\}\}\Big\}
\end{align*}
\end{definition}

Observe that $c_\infty(T)$ is the limit of $c_p(T)$ as $p$ goes to infinity. 
The Wasserstein distance is defined as the infimum of the costs of all transportation plans. 
Note that there is always at least one possible transportation plan as we can choose $\hat{S}^\mathcal{P}$ and $\hat{S}^\mathcal{Q}$ to be empty.

\begin{definition}
Fix $p\in [1,\infty)$. Let $\Theta$ be a totally ordered set and $\P=\bigoplus_{a\in S^\mathcal{P}} \mathcal{I}_a$ and $\mathcal{Q}=\bigoplus_{b\in S^\mathcal{Q}}\mathcal{I}_b$ be persistence modules over the ordered metric space $(\Theta, \leq, \d)$. 
The \emph{$p$-Wasserstein} distance between $\P$ and $\mathcal{Q}$ is
$$W_p(\mathcal{P}, \mathcal{Q})=\inf \{c_p(T)  \;\vert\; T \text{ a transportation plan between } \mathcal{P} \text{ and } \mathcal{Q}\}.$$
The \emph{bottleneck} distance between $\mathcal{P}$ and $\mathcal{Q}$ is
$$W_\infty(\mathcal{P}, \mathcal{Q})=\inf \{c_\infty(T) \;\vert\; T \text{ a transportation plan between } \mathcal{P} \text{ and } \mathcal{Q}\}.$$
\end{definition}

This definition agrees with the standard definitions of Wasserstein and  bottleneck distances between persistence diagrams when $\Theta$ is the real line with its standard order,  $\d(s,t)=|s-t|$, and $\Eph=\{[t,t]: t\in \R\}$. 
More generally, for any totally ordered metric space and any choice for the set of ephemeral intervals, the Wasserstein distance defined above will determine an extended metric. 
Again, for general persistence modules this will be, strictly speaking, a pseudo-distance. 
But, as discussed earlier, in this paper the persistence modules will only contain appropriate half-open intervals and $W_p(\mathcal{P},\mathcal{Q})$ satisfies the identity of indiscernibles.

\subsection{Wasserstein distance for extended persistence} % modules}

The Wasserstein distance between persistence modules is specified by the ordered metric space and set of ephemeral interval modules. Recall  from Section~\ref{ssec:ext_pers}  that extended persistent modules have parameter set $\Theta=O\cup R$, with $O= \{(t,\text{Ord}): t\in \R\} $ and $R= \{(t,\text{Rel}): t\in \R\}$,  and the total order over $P$ is
\begin{align*}
(s,\text{Ord})&<(t,\text{Ord}) \text{ when }s<t\\
(s,\text{Rel})&<(t,\text{Rel})  \text{ when }s>t\\
(s,\text{Ord})&<(t,\text{Rel}) \text{ for all }s,t.
\end{align*}
We make $\Theta$ an ordered metric space by constructing an appropriate extended metric over $\Theta$. 
A natural choice is
\begin{align*}
\dist((s,\text{Ord}),(t,\text{Ord}))&=|s-t|  \text{ for all }s,t.\\
\dist((s,\text{Rel}),(t,\text{Rel}))&=|s-t|  \text{ for all }s,t.\\ 
\dist((s,\text{Ord})(t,\text{Rel}))&=\infty \text{ for all }s,t.
\end{align*}

We also need to define the set of ephemeral interval modules; there are three different types: ordinary, relative and essential. 
We set 
\begin{align*}
\Eph=&\{\mathcal{I}_{\left[(t, \Ord), (t,\Ord)\right)}: t\in \R\} \cup \{\mathcal{I}_{\left[(t, \Rel), (t, \Rel)\right)}:t\in \R\}\\
& \cup \{\mathcal{I}: \mf{b}(\mathcal{I})=(t, \Ord) \text{ and } \mf{d}(\mathcal{I})=(t, \Rel) \text{ for some }t\in \R\}.
\end{align*}

For computational purposes it is much easier to split the calculation of distances between extended persistence modules into separate calculations for the submodules of the types $\Ord$, $\Rel$, $\Ess^+$ and $\Ess^-$. This is justified by the following proposition.

\begin{proposition}\label{prop:sumtypes}
Let $\mathcal{P}$ and $\mathcal{Q}$ be extended persistence modules in a single homology dimension and let $\mathcal{P}=\Ord(\mathcal{P})\oplus \Rel(\mathcal{P})\oplus\Ess^+(\mathcal{P}) \oplus\Ess^-(\mathcal{P})$ and $\mathcal{Q}=\Ord(\mathcal{Q})\oplus \Rel(\mathcal{Q})\oplus\Ess^+(\mathcal{Q}) \oplus\Ess^-(\mathcal{Q})$ be their decomposition into the four types of classes. Then 
\begin{align*}
W_p(\mathcal{P}, \mathcal{Q})^p=& W_p(\Ord(\mathcal{P}), \Ord(\mathcal{Q}))^p +W_p(\Rel(\mathcal{P}), \Rel(\mathcal{Q}))^p\\
&+W_p(\Ess^-(\mathcal{P}), \Ess^-(\mathcal{Q}))^p+W_p(\Ess^+(\mathcal{P}), \Ess^+(\mathcal{Q}))^p
\end{align*}
 for $p\in [1,\infty)$ and 
\begin{align*}
W_\infty(\mathcal{P}, \mathcal{Q})= \max\Big\{ &W_\infty(\Ord(\mathcal{P}), \Ord(\mathcal{Q})),W_\infty(\Rel(\mathcal{P}), \Rel(\mathcal{Q})),\\
&W_\infty(\Ess^-(\mathcal{P}), \Ess^-(\mathcal{Q})),W_\infty(\Ess^+(\mathcal{P}), \Ess^+(\mathcal{Q}))\Big\}.
\end{align*}
\end{proposition}

\begin{proof}
The right hand side of the both equations is the infimum of transportation costs over the set of transportation plans which never match any intervals of different types. 
It is thus sufficient to show that for any transportation plan between $\mathcal{P}$ and $\mathcal{Q}$ there is another transportation $T$ with the same or lesser cost such that any matched pair within $T$ keeps to the same type. Any two intervals of different types of $\Rel$, $\Ord$ or $\Ess$ are an infinite distance apart. Since every interval module has finite distance to some ephemeral interval it will always be more efficient to change any interval  that is matched to a different type to instead be unmatched. 
Similarly there is a  higher cost to match positive with negative essential classes than to leave both unmatched. \end{proof}

It is worth observing that in previous work, such as  \cite{carlsson2019parametrized,bauer2020universality}, the extended persistent homology modules are represented by multiple persistence diagrams, separating the different types into their own persistence diagrams. The ordinary persistence diagram has points above the diagonal, the relative persistence diagram has points only below the diagonal, and the essential persistence diagram has points on both sides --- positive above and negative below. 
The bottleneck distance in~\cite{bauer2020universality} is then defined as the formula within Proposition \ref{prop:sumtypes}. 

\begin{remark}
We believe that the Wasserstein distance could also be defined analogous to the algebraic Wasserstein distance in \cite{skraba2020wasserstein} but adapted to extended persistence, and that these two versions of Wasserstein distances would be equivalent. Given the enormous homological algebra set up required to prove such a result it is beyond the scope of this paper and left as a future direction of research.
\end{remark}

\section{Morse theory for manifolds with boundary and extended persistence}
\label{sec:morse} 
This section contains the main theoretical results relating extended persistence of a height function over a manifold with boundary to that of the same function restricted to the boundary. 
We establish these results using Morse theory, a standard technique when working with persistence modules built from sublevel set filtrations.  
Previous results, however, apply only to functions on manifolds not to those with boundary.
The presence of a boundary requires extra analysis to characterise critical points located on this boundary. 
We start by summarising the necessary definitions and results from Morse theory covering both the smooth and piecewise-linear settings.

\subsection{Background: Smooth and PL Morse theory}
\label{sec:MorseSetup}

We need our results about extended persistent homology to hold for both the smooth (theoretical) case, and the piecewise-linear setting relevant to numerical computations.
Most of the theorems and their proofs are effectively the same but we must first set up the definitions and relevant lemmas about critical points.  
The background theory is covered for the smooth case in \cite{braess1974morse, jankowski1972functions}, and the piecewise linear case in \cite{grunert2019pl}. We direct readers interested in more details to these references.   

Although regular and critical points and their indices in Morse theory are more commonly defined in terms of the derivatives and  Hessian of a function, this approach does not translate well to the PL setting. 
There is, however, an equivalent approach to defining critical points and indices that uses polynomial functions over charts, and this can be easily adapted to the PL setting. 
To make this paper self-contained we start by recalling the definitions of smooth and PL manifold (with or without boundary) in terms of charts.

\begin{definition}
For a topological space, $M$, and an open subset $U \subset M$, a \emph{chart} is a homeomorphism $\phi:U\to \phi(U)$ where $\phi(U)$ is a subset of Euclidean space. 
An \emph{atlas} for $M$ is an indexed family of charts $\{(U_\alpha, \phi_\alpha\})$  that cover $M$, i.e., $\cup U_\alpha = M$. 
A \emph{topological $n$-manifold} is a second countable, Hausdorff space equipped with an atlas where the codomain of each $\phi_{\alpha}$ is an open subset of $\R^n$.   
A \emph{topological $n$-manifold with boundary} is a second countable, Hausdorff space equipped with an atlas where the codomain of each $\phi_{\alpha}$ is an open subset of $\left[0,\infty\right) \times\R^{n-1}$.\end{definition}

To introduce the adjectives smooth and piecewise linear (PL) we need to discuss the compatibility of $\phi_\alpha$ and $\phi_\beta$ on the intersections of their domains. 
Given two charts $(U_\alpha, \phi_\alpha)$ and $(U_\beta, \phi_\beta)$ where $U_\alpha \cap U_\beta$  has non-empty intersection we can define two different maps by restricting the domains of $\phi_\alpha$ and $\phi_\beta$ to $U_\alpha \cap U_\beta$.
The new homeomorphisms are $\phi_\alpha \circ (\phi_\beta)^{-1}: \phi_{\beta}(U_\alpha \cap U_\beta) \to \phi_\alpha (U_\alpha \cap U_\beta)$ and $\phi_\beta \circ (\phi_\alpha)^{-1}: \phi_{\alpha}(U_\alpha \cap U_\beta) \to \phi_\beta (U_\alpha \cap U_\beta)$. These are called the \emph{transition maps} between charts. 

\begin{definition}
A topological $n$-manifold, with or without boundary, is called \emph{smooth} if its transition maps are smooth. It is called \emph{piecewise-linear} (PL for short) if its transition maps are piecewise-linear.
\end{definition}

We say that $\{(U_\alpha, \phi_\alpha)\}$ is \emph{maximal} if there does not exist another atlas containing it with more charts. 
A maximal atlas is often referred to as the smooth structure, or respectively, the PL structure of a manifold. 
Once we have a smooth (or PL) structure we can define what it means for a function $f:M \to \R$ to be smooth or piecewise linear. 

\begin{definition}
Let $M$ be a smooth $n$-manifold, with or without boundary, with smooth (respectively PL) structure $\{(U_\alpha, \phi_\alpha)\}$. A function $f:M\to \R$ is \emph{smooth} (respectively \emph{PL}) if $\phi_\alpha^{-1}\circ f: \phi_\alpha(U_\alpha)\to \R$ is smooth (respectively \emph{PL}) for all charts $(U_\alpha, \phi_\alpha)$.
\end{definition}

An example to keep in mind is $M$ being a smooth or piecewise linear $n$-dimensional subset of $\R^d$ with its structure inherited from the embedding.  
A simple function on such a manifold is the height function with respect to some unit vector $v\in S^{d-1}$, i.e., $f(x)=v\cdot x$.

The classical approach to defining critical points in Morse theory is as follows. For a manifold $M$ without boundary and a smooth function $f:M \to \R$. Let $p\in M$ and choose a chart $(U, \phi)$ with $p\in U$. We say that $p\in M$ is a \emph{critical point} of $f$ if $d(f\circ{\phi}^{-1})(\phi(p))=0$.  A critical point is \emph{non-degenerate} if the Hessian of $f\circ \phi$ at $p$ is non-singular. We then say the \emph{Morse index of f at $p$} is the number of negative eigenvalues of the Hessian, counting multiplicity. A point is \emph{regular} if it is not critical. These definitions are well defined as they do not depend on the choice of chart (see \cite{Milnor}).

Instead of using definitions for critical and regular points in terms of the derivative, we need an alternative that will be more adaptable to the PL setting. 
By using the implicit function theorem we can redefine regular points by the existence of a linear function over some chart.
We can also remove the need to reference the Hessian for defining the index of a critical point by using the Morse Lemma.

\begin{lemma}[Morse Lemma]
Let $M$ be a smooth $n$-manifold without boundary and  $f:M \to \R$ a smooth function. The point $p\in M$ is a regular point of $f$ if and only if there is a chart $(U, \phi)$ where $\phi(p)=0$ and
$$f\circ \phi^{-1}(x_1, x_2, \ldots x_n)=f(p)+x_n$$
in some neighbourhood of $0$.

The point $p\in M$ is a non-degenerate critical point of $f$ with Morse index $k$ if and only if there is a chart $(U, \phi)$ where $\phi(p)=0$ and $$f\circ \phi^{-1}(x_1, x_2, \ldots x_n)=f(p)-(x_1)^2 - (x_2)^2 -\ldots -(x_k)^2 +(x_{k+1})^2 +\ldots + (x_n)^2$$ in a neighbourhood of $0$. 
\end{lemma}

The proof of this lemma is covered in \cite{Milnor}. We use it as an equivalent definition of a regular point and a non-degenerate critical point of Morse index $k$. 
In the piecewise linear setting  the only modification is to replace squares with absolute values.

\begin{definition}
Let $M$ be an  $n$-dimensional  PL manifold without boundary and  $f:M \to \R$ a PL function. 
The point $p\in M$ is a \emph{regular} point of $f$ if and only if there is a chart $(U, \phi)$  containing $p$ of the form $$f\circ \phi^{-1}(x_1, x_2, \ldots x_n)=f(p)+x_n.$$
The point $p\in M$ is a non-degenerate critical point of $f$ with Morse index $k$ if and only if there is a chart $(U, \phi)$ with $\phi(p)$ which is of the form 
$$f\circ \phi^{-1}(x_1, x_2, \ldots x_n)=f(p)-|x_1| - |x_2| -\ldots -  |x_k| +|x_{k+1}| +\ldots + |x_n|$$ in a neighbourhood of $0$. 
\end{definition}

We now need to generalise the definitions of regular and critical points to the case of a function over a manifold with boundary $(M, \partial M)$.
Points in the interior of $M$ are treated exactly as above, so we need only discuss the case for points on the boundary. 
We again phrase the definitions using charts to make it easy to move between smooth and PL settings, following the terminology and notation in \cite{grunert2019pl}. 
Recall that a chart containing a point, $p \in \partial M$ is homeomorphic to a subset of 
$\{(x_1, x_2, \ldots, x_n)\in \R^n \;\vert\; x_1 \geq 0\}$, with $\phi(p) = (0, x_2, \ldots, x_n)$. 

\begin{definition}
Let $(M, \partial M)$ be a smooth (respectively PL) $n$-manifold with boundary and  $f:M \to \R$ a smooth (respectively PL) function. The point $p\in \partial M$ is a \emph{regular} point of $f$ if and only if there is a chart $(U, \phi)$ with $\phi(p)=0$ of the form $f\circ \phi^{-1}(x_1, x_2, \ldots x_n)=f(p)+x_n$. 
\end{definition}

A point on the boundary is critical if it is critical for $f$ restricted to $\partial M$, but the definition of its index requires additional information about whether the function increases or decreases as we move into the manifold.

\begin{definition}
Let $(M, \partial M)$ be a smooth $n$-manifold with boundary and $f:M \to \R$ a smooth function. 
The point $p\in \partial M$ is a \emph{non-degenerate critical} point of $f$ with index $(k,\eta)$ if only if there is a chart $(U, \phi)$ with $\phi(p)=0$ such that 
$$f\circ \phi^{-1}(x_1, x_2, \ldots x_n)=f(p) + \eta x_1 - (x_2)^2 -\ldots -(x_{k+1})^2 +(x_{k+2})^2 +\ldots + (x_{n})^2.$$
The second term of the index, $\eta \in \{-1, 1\}$,  defines the sign of the critical point: 
if $\eta = 1$ we say that $p$ is $(+)$-critical, and if $\eta = -1$, then $p$ is $(-)$-critical. 
\end{definition} 

The analogous definition for the piecewise linear case is: 
\begin{definition}
Let $(M, \partial M)$ be a PL $n$-manifold with boundary and  $f:M \to \R$ a PL function. The point $p\in \partial M$ is a \emph{non-degenerate critical} point of $f$ of index $(k,\eta)$ if there is a chart $(U, \phi)$ with $\phi(p)=0$ of the form 
$$f\circ \phi^{-1}(x_1, x_2, \ldots x_n)=f(p)+\eta x_1 - |x_2| -\ldots -|x_{k+1}| +|x_{k+2}| +\ldots + |x_{n}|.$$
Again, $p$ is $(+)$-critical when $\eta=+1$ and $(-)$-critical when $\eta=-1$.
\end{definition}

Please note that there is inconsistency within the literature in terms of sign conventions for critical points on the boundary and our choice may differ from sources the reader is familiar with.

Now we have the definitions for all the different types of critical point, we can define what a Morse function is for both the smooth and PL settings.

\begin{definition}
Given a smooth (respectively PL) manifold with boundary $(M, \partial M)$, we say that $f:M\to \R$ is a \emph{Morse function} if
\begin{itemize}
\item $f$ is smooth (respectively PL)
\item None of the critical points of $f \vert_{\Int(M)}$ and $f \vert_{\partial M}$ are degenerate. 
\item All the critical values for $f|_{\Int(M)}$ and $f|_{\partial M} $ combined are distinct and finite in number.
\end{itemize}
\end{definition}

In the following we describe the (persistent) homology in terms of the signs of critical points so it is useful to have notation for this. 

\begin{definition}
Suppose $f:(M,\partial M)\to \R$ is a Morse function. Let $\Crit(f, k)$ denote the set of index-$k$ critical points of $f$;  these points must lie in the interior of $M$.  Let $\Crit(f, (k ,\eta))$ denote the set of critical points of $f|_{\partial M}$ with index $(k,\eta)$. If $p \in \partial M$ is a critical point of $f|_{\partial M}$, with index $(k,\eta)$ denote the sign of $p$ by $\sgn (f,p) = \eta$. 
\end{definition}

Highly analogous to the well-known theory of Morse functions on manifolds, we can use the index of critical points to compute the relative homology of nearby sublevel sets of $f$. 

\begin{proposition}\label{prop:strongMorse}
Let $(M, \partial M)$ be a smooth (respectively PL) manifold with boundary and $f: M \to \R$ a smooth (respectively PL) Morse function.  We consider homology with coefficients in a field $\mathbb{F}$, and use Kronecker delta notation $\delta_i^k$ below. 
\begin{itemize}
\item If $t$ is not a critical value of neither $f$ nor $f|_{\partial M}$ then $H_i(M_{t+\epsilon}, M_{t-\epsilon})=0$  for all $i$ and all $\epsilon >0$ sufficiently small.
\item If $p\in \Crit(f, k)$ then $H_i(M_{f(p)+\epsilon},M_{f(p)-\epsilon})=\delta_i^k \, \mathbb{F}$ for all $i$ and for all $\epsilon >0$ sufficiently small. 
\item If $p\in \Crit(f, (k,-1))$ then $H_i(M_{f(p)+\epsilon},M_{f(p)-\epsilon})=0$ for  $\epsilon >0$ sufficiently small.
\item If $p\in \Crit(f, (k,+1))$ then $H_i(M_{f(p)+\epsilon}, M_{f(p)-\epsilon})=\delta_i^k \, \mathbb{F}$  for all $i$, for  $\epsilon >0$ sufficiently small.
\end{itemize}
\end{proposition}

For the smooth case, this proposition is proved in \cite{braess1974morse} and in \cite{jankowski1972functions}. Please note that in  \cite{jankowski1972functions} they use the term ``m-function'' for Morse function.  Some minor massaging is needed to convert their results to the homology statements above as they describe the changes in terms of glueing cells. 
The PL version of this proposition is proved in \cite{grunert2019pl}. 

We can determine critical points and indices of $(-f)$ from those of $f$ using charts, as summarised in the following lemma which holds for both the smooth and PL settings. 

\begin{lemma}\label{lem:negcrit}
Let $(M,\partial M)$ be an $n$-manifold with boundary and $f:M\to \R$ a Morse function. 
Then $(-f):M \to \R$ is also a Morse function with $\Crit((-f),k)=\Crit(f,n-k)$ and $\Crit((-f), (k,\eta))=\Crit(f, (n-k-1, -\eta))$ for $\eta=\pm 1$. 
\end{lemma}

This facilitates analogous homology results as in Proposition \ref{prop:strongMorse} but for relative homology of superlevel sets.

\begin{corollary}\label{cor:MorseRel}
Let $(M, \partial M)$ be a smooth (respectively PL) $n$-manifold with boundary and $f: M \to \R$ a smooth (respectively PL) Morse function. 
\begin{itemize}
\item If $t$ is not a critical value of neither $f$ nor $f|_{\partial M}$ then $H_i(M^{t-\epsilon}, M^{t+\epsilon})=0$  for all $i$ and all $\epsilon >0$ sufficiently small .
\item If $p\in \Crit(f, n-k)$ then $H_i(M^{f(p)-\epsilon},M^{f(p)+\epsilon})=\delta_i^{k} \, \mathbb{F}$ for all $i$ and for all $\epsilon >0$ sufficiently small. 
\item If $p\in \Crit(f, (n-k-1,+1))$ then $H_i(M^{f(p)-\epsilon},M^{f(p)+\epsilon})=0$ for all $i$ and for  $\epsilon >0$ sufficiently small.
\item If $p\in \Crit(f, (n-k-1,-1))$ then $H_i(M^{f(p)-\epsilon}, M^{f(p)+\epsilon})=\delta_i^{k} \, \mathbb{F}$  for all $i$ and for  $\epsilon >0$ sufficiently small.
\end{itemize}
\end{corollary}

\begin{proof}
We first want to write the superlevel sets of $f$ in terms of sublevel sets of $(-f)$. 
We have  $M^{s}=(-f)^{-1}(-\infty, -s]$, and thus
$$H_i(M^{f(p)-\epsilon}, M^{f(p)+\epsilon})=H_i\big((-f)^{-1}(-\infty, (-f)(p)+\epsilon], (-f)^{-1}(-\infty, (-f)(p)-\epsilon] \big).$$

If $t$ is not a critical value for $f$ nor $f|_{\partial M}$ then by Lemma \ref{lem:negcrit} $(-t)$ is not a critical value of $(-f)$ nor $(-f)|_{\partial M}$. By Proposition \ref{prop:strongMorse} we know $$H_i((-f)^{-1}(-\infty, -t+\epsilon], (-f)^{-1}(-\infty, -t -\epsilon])=0$$  for all $i$ and all $\epsilon >0$ sufficiently small.

If $p\in \Crit(f, n-k)$, then by Lemma \ref{lem:negcrit} we have $p\in \Crit((-f),k)$. If $p\in \Crit(f,(n-k-1, -1))$ then by Lemma \ref{lem:negcrit} $p\in \Crit((-f),(k,+1))$.
In both cases we can apply Proposition \ref{prop:strongMorse}, with $(-f)$ at $p$, which implies that 
\begin{align*}
H_i\big((-f)^{-1}(-\infty, (-f)(p)+\epsilon], (-f)^{-1}(-\infty, (-f)(p)-\epsilon]\big)=\delta_i^{k} \, \mathbb{F}.
\end{align*}

If $p\in \Crit(f, (n-k-1,+1))$, then by Lemma \ref{lem:negcrit} we have $p\in \Crit((-f),(k,-1))$. By Proposition \ref{prop:strongMorse}, with $(-f)$ at $p$, we know 
\begin{align*}
H_i\big((-f)^{-1}(-\infty, (-f)(p)+\epsilon], (-f)^{-1}(-\infty, (-f)(p)-\epsilon]\big)=0.
\end{align*}
for  $\epsilon >0$ sufficiently small.
\end{proof}

As might be expected, there is a direct relationship between the critical values of Morse functions and the endpoints of intervals in the barcode decomposition of extended persistent homology. We will need to distinguish between endpoints lying in the ordinary and relative parameter spaces  as they behave differently.

Let $\XPH(M,f)$ be the extended persistence module constructed from $f:M\to \R$. 
To ease notation let $$\bt^{\ord}_k(M,f):=\mf{b}(\Ord_k(M,f)\oplus\Ess_k(M,f))$$ and $$\bt^{\rel}_k(M,f):=\mf{b}(\Rel_k(M,f))$$ These are the sets of parameters $\{(t, \ord)\}$ and $\{(t, \rel)\}$ respectively where a new interval begins in the interval decomposition of $\XPH(M,f)$. 
Similarly let $$\dt^{\ord}_k(M,f):=\mf{d}(\Ord_k(M,f))$$ and $$\dt^{\rel}_k(M,f):=\mf{d}(\Rel_k(M,f)\oplus\Ess_k(M,f)).$$ These are the sets of parameters $\{(t, \ord)\}$ and $\{(t, \rel)\}$ respectively where an interval finishes in the interval decomposition of $\XPH(M,f)$. 
Furthermore let $\bt_k(M,f)=\bt_k^{\ord}(M,f)\cup \bt_k^{\rel}(M,f)$ and $\dt_k(M,f)=\dt_k^{\ord}(M,f)\cup \dt_k^{\rel}(M,f)$ denote the sets of birth and death parameters respectively for the extended persistence module $\XPH(M,f)$. 
%Note here that this includes all endpoints for any of the ordinary, relative and essential classes. 
In constructing these sets we use the fact that every essential class is born somewhere in the ordinary parameter range and then dies somewhere in the relative parameter range.

The following corollary follows from  Proposition \ref{prop:strongMorse} and Lemma \ref{cor:MorseRel}.
%\VR{We've not explicitly defined a ``strong Morse function''.}
\begin{cor}\label{lem:ord}
Let $(M,\partial M)$ be an $n$-dimensional manifold with boundary and let $f:M\to \R$ be a Morse function. Then
$$\bt^{\ord}_k(M,f)\cup \dt^{\ord}_{k-1}(M,f)=\{(f(p),\ord)| \; p\in \Crit(f,k)\cup \Crit(f,(k,+1))\}.$$
and
$$\bt^{\rel}_k(M,f)\cup \dt^{\rel}_{k-1}(M,f)=\{(f(p),\rel)| \; p\in \Crit(f,n-k)\cup \Crit(f,(n-k-1,-1))\}.$$
%$x\in \Int (M)$ is a critical point of index $j$ or $x\in \partial M$ is a critical point of index $(j, -1)$.
\end{cor}

\subsection{Relating the extended persistent homology of a manifold to that of its boundary}
\label{subsec:XPH_boundary}

We can now restrict to the situation of interest for the $\XPHT$; that of computing the extended persistent homology of a height function over a compact $n$-dimensional manifold with boundary embedded in $\R^n$.  
The results in this section start by comparing the sets of birth and death parameters for the height filtration of the manifold and for its boundary, in Propositions~\ref{prop:critM} and~\ref{prop:sametimes}. 
The next step is to show that these births and deaths are paired consistently as endpoints of intervals in the relevant persistence modules (Theorem~\ref{thm:oplus}). 
We finish with a complete characterisation of the extended persistent homology for the manifold as a  submodule of that for its boundary in Theorem~\ref{thm:submodule}.  

The height function is specified in a direction $v$ and restricted to various subsets of $\R^n$. 
That is, $h_v:\R^n\to \R$ with $h_v(x)=x\cdot v$. To ease notation let $h_v^S$ denote the restriction of the height function to $S\subset \R^n$, that is $h_v^S=h_v|_S$.

\begin{proposition}\label{prop:critM}
Let $M\subset \R^n$ be a compact $n$-manifold with boundary.  Suppose that $h_v^M:M \to \R$, the height function in direction $v$, is a Morse function.  For each critical value $t$ let $p(t)$ be the unique critical point of $h^M_v$ or $h_v^{\partial M}$ with $h_v(p)=t$. For all $k>0$ we have 

\begin{align*}
\bt_k^{\ord}(M,h_v^M)&=\{(t,\ord)\in\bt_k^{\ord}(\partial M, h_v^{\partial M}): \sgn(h_v^M,p(t))=+1)) \}\\
\bt_k^{\rel}(M,h_v^M)&= \{(t, \rel)\in \bt_k^{\rel}(\partial M, h_v^{\partial M}):\sgn(h_v^M,p(t))=-1\}\\
\dt_k^{\ord}(M,h_v^M)&=\{(t,\ord)\in\dt_k^{\ord}(\partial M, h_v^{\partial M}): \sgn(h_v^M,p(t))= +1\} \\
\dt_k^{\rel}(M,h_v^M)&= \{(t, \rel)\in \dt_k^{\rel}(\partial M, h_v^{\partial M}):\sgn(h_v^M,p(t))=-1\}\\
\end{align*}

\end{proposition}
\begin{proof}
Choose $R>0$ large enough so that  $M\subset B(0,R)$ where $B(0,R)$ is the open ball of radius $R$ centred on the origin. Let $L=\overline{B(0,R)}\backslash \Int(M)$.  As there are only finitely many critical points of $h_v^M$ and $h_v^{\partial M}$ there is an $\epsilon>0$ such that all the critical values are at least $\epsilon$ apart.  The critical values lie within $[\inf(h_v(M)),\sup(h_v(M))]\subset (-R, R)$ so we can  reduce $\epsilon$ to be small enough that no critical value is within $\epsilon$ of $-R$ or $R$. 

The function $h_v$ defined over all of $\R^n$ has no critical points, so there will be no critical points in the interior of $M$.  This means we need only consider critical points of $h_v^{\partial M}$. 

For each $s\in \R$ we consider the sublevel sets of $h_v$ restricted to the three subsets: $M_s$,
$(\partial M)_s$
and $L_s$.
By construction $M_s\cap L_s=(\partial M)_s$ and $M_s\cup L_s= h_v^{-1}(-\infty, s]\cap \overline{B(0,R)}$.  
For each $k>0$  we therefore have $H_{k+1}(M_s\cup L_s)=0=H_k(M_s\cup L_s)$. 
Using this in the Mayer-Vietoris sequence shows us that  
$H_k((\partial M)_s)$ and $H_k(M_s)\oplus H_k(L_s)$ are isomorphic and hence 
\begin{align}\label{eq:sumBetti}
\beta_k(M_s) + \beta_k(L_s)&=\beta_k((\partial M)_s)
\end{align}
for all $s\in \R.$ 
For $k=0$ we know $H_0(M_s\cup L_s)=1$ whenever $s\geq -R$. 
Mayer-Vietoris then gives the short exact sequence
$$0\to H_0((\partial M)_s)\to H_0(M_s)\oplus H_0(L_s)\to H_0(M_s\cup L_s) \to 0.$$
By comparing the ranks we have
\begin{align}\label{eq:sumBetti0}
\beta_0(M_s) +\beta_0(L_s)&=\beta_0((\partial M)_s)+1 
\end{align}
whenever $s>-R$.
 
Suppose that $(t,\ord)\in \bt_k^{\ord}(M,h_v)$ and thus $\beta_k(M_{t+\epsilon})-\beta_k(M_{t-\epsilon})=1$. By Proposition \ref{prop:strongMorse} we know 
$\sgn(h_v^M,p(t))=+1$ and this  implies $\sgn(h_v^L,p(t))=-1$.
Proposition \ref{prop:strongMorse} now implies that $\beta_k(L_{t+\epsilon})=\beta_k(L_{t-\epsilon})$. 
For $k>0$ we can use \eqref{eq:sumBetti} to calculate 
\begin{align*}
\beta_k((\partial M)_{t+\epsilon})&-\beta_k((\partial M)_{t-\epsilon}) \\
&=(\beta_k(M_{t+\epsilon})+\beta_k(L_{t+\epsilon}))-(\beta_k(M_{t-\epsilon})+\beta_k(L_{t-\epsilon}))\\
&=1.
\end{align*}
If $k=0$ we instead use \eqref{eq:sumBetti0} to calculate
\begin{align*}
\beta_0((\partial M)_{t+\epsilon})&-\beta_0((\partial M)_{t-\epsilon})\\
&=(\beta_0(M_{t+\epsilon})+\beta_0(L_{t+\epsilon})-1)-(\beta_0(M_{t-\epsilon})+\beta_0(L_{t-\epsilon})-1)\\
&=1.
\end{align*}
This is where we use the requirement that $\epsilon$ is small enough that all critical points of $M$ are  greater than $-R+\epsilon$.
Since $h_v^{\partial M}$ is Morse and $t$ is the only critical value of $h_v^{\partial M}$ in $[t-\epsilon,t+\epsilon]$ we thus conclude that there is a birth event at $t$, that is $(t,\ord)\in \bt_k^{\ord}(\partial M, h_v)$.

Now suppose that $(t,\ord)\in \bt_k^{\ord}(\partial M,h_v)$ with $\sgn(h_v^M,p(t))=+1$. 
This means $\beta_k((\partial M)_{t+\epsilon})-\beta_k((\partial M)_{t-\epsilon})=1$, and  $\sgn(h_v^L,p(t))=-1$. Proposition \ref{prop:strongMorse} again tells us that $\beta_k(L_{t+\epsilon})=\beta_k(L_{t-\epsilon})$ and using \eqref{eq:sumBetti} we calculate
\begin{align*}
\beta_k(M_{t+\epsilon})&-\beta_k(M_{t-\epsilon}) \\
&=(\beta_k((\partial M)_{t+\epsilon})-\beta_k(L_{t+\epsilon}))-(\beta_k((\partial M)_{t-\epsilon})-\beta_k(L_{t-\epsilon}))\\
&=1.
\end{align*}
If $k=0$ then we instead use \eqref{eq:sumBetti0} to calculate
\begin{align*}
\beta_0(M_{t+\epsilon})&-\beta_0(M_{t-\epsilon})\\
&=(\beta_k((\partial M)_{t+\epsilon})-\beta_k(L_{t+\epsilon})+1) - (\beta_k((\partial M)_{t-\epsilon})-\beta_k(L_{t-\epsilon})+1)\\
&=1.
\end{align*}

We have again used $t-\epsilon> -R$.  Since $t$ is the only critical value of $h_v^{\partial A}$ in $[t-\epsilon,t]$ we conclude that $(t,\ord)\in \bt_k^{\ord}(M, h_v)$.

When considering the sets of births and deaths in the relative parameter range we need to use a relative version of the Mayer-Vietoris sequence. For this recall that $M\cap L =\partial M$, and $M^s\cap L^s=\partial M^s$. 
The relative version of the Mayer-Vietoris sequence states that there is a long exact sequence
\begin{align*}
\cdots \to H_{k+1}(M\cup L, &M^s\cup L^s) \to H_k((\partial M),(\partial M)^s)\to \cdots \\
   \to &H_k(M,M^s)\oplus H_k(L,L^s)\to H_{k}(M\cup L, M^s\cup L^s)\to\cdots 
\end{align*}  
Since $M\cup L=B(0,R)$, and $M^s\cup L^s=B(0,R)\cap h_v^{-1}[s,\infty)$ we have (for $k\geq0$ this time) $H_{k+1}(M\cup L, M^s\cup L^s)=0=H_{k}(M\cup L,M^s\cup L^s)$ for $s < R$. 
This implies $H_k((\partial M),(\partial M)^s)$ and $H_k(M,M^s)\oplus H_k(L,L^s)$ are isomorphic and hence 
$$\beta_k(\partial M,(\partial M)^s)=\beta_k(M,M^s)+\beta_k(L,L^s)$$
 for all $s<R$.

Suppose $(t,\rel)\in \bt_k^{\rel}(M,h_v^M)$ and thus $\beta_k(M,M^{t-\epsilon})-\beta_k(M,M^{t+\epsilon})=1$. 
As $(t,\rel)\in \bt_k^{\rel}(M,h_v^M)$ we have $\sgn(h_v^M,p(t))=-1$ we thus $\sgn(h_v^L,p(t))=+1$. 

From Corollary \ref{cor:MorseRel} we know that  $\beta_k(L,L^{t-\epsilon})=\beta_k(L,L^{t+\epsilon})$. We then calculate
\begin{align*}
\beta_k(\partial M,(\partial M)^{t-\epsilon})&-\beta_k(\partial M,(\partial M)^{t+\epsilon})\\
&=(\beta_k(M,M^{t-\epsilon})+\beta_k(L,L^{t-\epsilon}))-(\beta_k(M,M^{t+\epsilon})+\beta_k(L,L^{t+\epsilon}))\\
&=1.
\end{align*}
Since $t$ is the only critical value of $h_v^{\partial M}$ in $[t-\epsilon,t+\epsilon]$ we  conclude that $(t,\rel)\in \bt_k^{\rel}(\partial M, h_v)$.

Now suppose that $(t,\rel)\in \bt_k(\partial M,h_v)$ with $\sgn(h_v^M,p(t))=-1$
These facts imply $\beta_k(\partial M,(\partial M)^t)-\beta_k(\partial M, (\partial M)^{t+\epsilon})=1$ and  $\sgn(h_v^L,p(t))=+1$. By Corollary \ref{cor:MorseRel} we therefore have  $\beta_k(L,L^{t-\epsilon})=\beta_k(L,L^{t+\epsilon})$ and can calculate
\begin{align*}
\beta_k(M,&M^{t-\epsilon})-\beta_k(M,M^{t+\epsilon})\\
&=(\beta_k(\partial M,(\partial M)^{t-\epsilon})-\beta_k(L,L^{t-\epsilon}))-(\beta_k(\partial M,(\partial M)^{t
+\epsilon})-\beta_k(L,L^{t+\epsilon}))\\
&=1.
\end{align*}
Since $t$ is the only critical value of $h_v^{\partial M}$ in $[t-\epsilon,t+\epsilon]$ we conclude that $(t,\rel)\in \bt_k^{\rel}(M, h_v)$.

The proof for the sets of death critical values is highly analogous; the difference of the Betti numbers is  $-1$ instead of $1$.
\end{proof}

Throughout the following collection of results we fix the following sets:
Let $A$ be a compact subset of $\R^n$ whose boundary $\partial A = X$ is therefore a finite collection of disjoint $n-1$ manifolds.  
Let $R>0$ such that $A\subset B(0,R)$. Let $B$ be the set such that $A\cup B=\overline{B(0,R)}$ and $A\cap B=X$. 

Let $S_R$ denote the sphere of radius $R$. We can observe that $\partial A=X$ and $\partial B=X\cup S_R$.
Let $h_v$ be the height function in the direction $v\in S^{n-1}$, with $v$ such that $h_v^X$ is a Morse function. 

\begin{proposition}\label{prop:sametimes}
Let $A\subset \R^n$ be a compact $n$-dimensional manifold with boundary. 
Let $h_v:\R^n \to \R$ be the height function in direction $v$ such that $h_v^X$ is a Morse function.
Let $R>0$ be such that $A\subset B(0,R)$. 
Let $B$ be the set such that $A\cup B=\overline{B(0,R)}$ and $A\cap B=X$. 
Let $S_R$ denote the sphere of radius $R$. Then we have the equality of the following disjoint unions:

\begin{align*}
\bt_0(X,v)\sqcup\{(-R,ord)\}&=\bt_0(A,v) \sqcup \bt_0(B,v)\\
\bt_k(X,v)&=\bt_k(A,v) \sqcup \bt_k(B,v)\qquad \text{ for }k>0\\
\dt_0(X,v)\sqcup\{(R,ord)\}&=\dt_0(A,v) \sqcup \dt_0(B,v)\\
\dt_k(X,v)&=\dt_k(A,v) \sqcup \dt_k(B,v)\qquad\text{ for }k>0
\end{align*}
\end{proposition}

\begin{proof}
Since $\partial A=X$ we can use Proposition \ref{prop:critM} to say 
\begin{align*}
\bt_k(A,h_v)=&\{(t,ord)\in \bt_k(X,h_v) | \; \sgn(h_v^A, p(t ))=+1 \} \\
&\cup \{(t,rel) \in \bt_k(X,h_v) | \; \sgn(h_v^A, p(t ))=-1 \}\\
\dt_k(A,h_v)= & \{(t,ord)\in \dt_k(X,h_v) | \; \sgn(h_v^A, p(t ))=+1 \}\\
& \cup \{(t,rel)\in \dt_k(X,h_v) | \; \sgn(h_v^A, p(t ))=-1 \}.
\end{align*}

Since $\partial B=X\sqcup S_R$ we can again apply Proposition \ref{prop:critM} (now with $B$ playing the role of $M$) to say 
\begin{align*}
\bt_k(B,h_v)=&\{(t,ord)\in \bt_k(X\sqcup S_R,h_v)| \sgn(h_v^B,p(t))= +1\} \\
&\cup \{(t,rel)\in \bt_k(X\sqcup S_R,h_v)| \sgn(h_v^B,p(t))= -1\}\\
\dt_k(B,h_v)= & \{(t,ord)\in \dt_k(X \sqcup S_R,h_v) \sgn(h_v^B,p(t))= +1\}\\
& \cup \{(t,rel)\in \dt_k(X\sqcup S_R ,h_v)| \sgn(h_v^B,p(t))= -1\}.
\end{align*}

The critical points of $h_v^B$ which lie on $S_R$ are well understood. There are two critical points; one birth in dimension $0$ at $p_1=-Rv$ with value $h_v(p_1)=-R$, and a death in dimension $0$ at $p_2=Rv$ with $h_v(p_2)=R$. We thus can rewrite the birth and death sets of $B$ as 
\begin{align*}
\bt_0(B,h_v)=\{(-R,ord)\}&\cup\{(t,ord)\in \bt_0(X,h_v)| \sgn(h_v^B,p(t))= +1\} \\
&\cup \{(t,rel)\in \bt_0(X,h_v)| \sgn(h_v^B,p(t))= -1\}\\
\dt_0(B,h_v)=\{(R,rel)\} &\cup \{(t,ord)\in \dt_0(X,h_v) \sgn(h_v^B,p(t))= +1\}\\
& \cup \{(t,rel)\in \dt_0(X,h_v)| \sgn(h_v^B,p(t))= -1\}.
\end{align*}
and for $k>0$ we have 
\begin{align*}
\bt_k(B,h_v)=&\{(t,ord)\in \bt_k(X,h_v)| \sgn(h_v^B,p(t))= +1\} \\
&\cup \{(t,rel)\in \bt_k(X,h_v)| \sgn(h_v^B,p(t))= -1\}\\
\dt_k(B,h_v)= & \{(t,ord)\in \dt_k(X,h_v) \sgn(h_v^B,p(t))= +1\}\\
& \cup \{(t,rel)\in \dt_k(X,h_v)| \sgn(h_v^B,p(t))= -1\}.
\end{align*}

Since every critical point $p(t)\in X$ must be either $(+)$-critical or $(-)$-critical, by taking the union we get the statement of the proposition.
\end{proof}

Propositions~\ref{prop:critM} and~\ref{prop:sametimes} have shown how the sets of birth and death parameters for $X$, $A$, and $B$ are related. 
The following theorem proves the much stronger result that the pairing of endpoints of the bars is consistent, and so we have isomorphisms between various extended persistence modules. 
This theorem is not a new result -- it was proved using Alexander duality in \cite{edelsbrunner_alexander_2012}.  
We believe our Morse-theoretic proof may be more readily adapted to other scenarios. 

\begin{theorem}\label{thm:oplus}
We have $$\XPH_k(X,v) \oplus \XPH_k(A\cup B,v)= \XPH_k(A,v)\oplus \XPH_k(B,v).$$ 
That is 
$\XPH_0(X,v))\oplus \mathcal{I}_{((-R,ord),(R,rel))}= \XPH_0(A,v)\oplus \XPH_0(B,v)$
and for $k>0$  we have
$\XPH_k(X,v))=\XPH_k(A,v)\oplus \XPH_k(B,v).$
\end{theorem}

\begin{proof}
Let us first consider the case where $k>0$. Since $X\subset A$ and $X\subset B$ we have an induced morphisms on persistence modules $$\varphi:\XPH_k(X,v)\to\XPH_k(A,v)\oplus \XPH_k(B,v).$$ Furthermore from the ordinary and relative versions of the Mayer-Vietoris sequence we know $\varphi_{(t,ord)}$, and $\varphi_{(t,rel)}$ are both isomorphisms for all $t\in \R$. This implies that $\varphi$ is must be injective. 

Injective morphisms between persistence modules were studied extensively in \cite{bauer2014induced}. Bauer and Lesnick showed that an injective morphism will induce a injective map $\rho$ on the sets of intervals in the interval decomposition of $\XPH_k(X,v)$ to those in the interval decomposition of $\XPH_k(A,v)\oplus \XPH_k(B,v)$ such that every interval in $[b,d) $ in $\XPH_k(X,v)$ is mapped to an interval $\rho([b,d))=[b',d)$ with the same death time and $b'\leq b$.

By Proposition \ref{prop:sametimes} we know that the sets of start and end parameters for the barcode decompositions satisfy  
$\bt_k(X,v)=\bt_k(A,v)\cup \bt_k(B,v)$. As the two persistence modules have the same number of intervals, the matching $\rho$ must in fact be a bijection. Observe that if $f:S\to S$ is a bijection from a finite set to itself such that $f(s)\leq s$ for all $s\in S$ then we are forced to have $f$ the identity. This argument shows $\rho([b,d))=\rho([b,d))$ and the interval decompositions of $\XPH_k(X,v)$ and $\XPH_k(A,v)\oplus \XPH_k(B,v)$ are the same and they are isomorphic as persistence modules.

For the case where $k=0$  we need to consider the complication of the homology class corresponding to the sphere $S_R$. 
We know from Proposition \ref{prop:sametimes} that $\bt_0(X,v)\sqcup\{(-R,ord)\}=\bt_0(A,v) \sqcup \bt_0(B,v)$, which we will denote $\bt$, and $\dt_0(X,v)\sqcup\{(R,ord)\}=\dt_0(A,v) \sqcup \dt_0(B,v)$, which we will denote $\dt$. This means that we can define a bijection $\rho:\bt\to \bt$ such that $\rho (b)=b'$ if there exists a $d$ such that $$(b,d]\in \XPH_0(X,v))\oplus \mathcal{I}_{((-R,ord),(R,rel))}$$ and $$(b',d]\in \XPH_k(A,v)\oplus \XPH_k(B,v).$$
Observe that $[(-R, ord), (R,rel))$ is an interval in the interval decomposition of $\XPH_0(B,v)$ - this corresponds to the connected component containing $S_R$. This implies that $\rho((-R, ord))=(-R,ord)$. 
Just as in the case for $k>0$ we can consider the ordinary and relative Mayer Vietoris sequences to show that the morphisms $H_0(X_t)\to H_0(A_t)\oplus H_0(B_t)$ and $H_0(X,X^t) \to H_0(A,A^t)\oplus H_0(B,B^t)$ induced by inclusions are injective for all $t$ and hence the morphism $\varphi:\XPH_0(X,v)\to\XPH_0(A,v)\oplus \XPH_0(B,v)$ is injective. Again this implies there is an injective map which pairs each interval $[b,d) $ in $\XPH_0(X,v)$  to an interval $[b',d)$ with the same death time and $b'\leq b$. This implies that our function $\rho:\bt\to \bt$ has $\rho(b)\leq b$ for all $b\in \bt_0(\XPH_0(X))$. Together these imply $\rho(b)\leq b$ for all $b\in \bt$ which, since $\bt$ is finite, implies $\rho$ is the identity. 
Hence the interval decompositions of $\XPH_0(X,v) \oplus \mathcal{I}_{[(-R,ord),(R,rel))}$ and $\XPH_0(A,v)\oplus \XPH_0(B,v)$ are the same and they are isomorphic as persistence modules.
\end{proof}

Combining Theorem \ref{thm:oplus} with Proposition \ref{prop:sametimes} allows us to express the extended persistent homology of a height function over $A$ as a nice submodule of the extended persistent homology of that same height over $\partial A$.

\begin{theorem}\label{thm:submodule}
Let $A\subset \R^n$ be an $n$-manifold with boundary $X=\partial A$. Let $v$ be a direction such that $h_v^A:A\to \R$ is a Morse function.
Let the interval decomposition of the $k$-dimensional extended persistent homology of $h_v^X:X\to \R$ be 
$$\XPH_k(X,h_v)=\bigoplus_{[b_i,d_i)\in S_X} \mathcal{I}_{[b_i,d_i)}.$$ 

Let $J_A^k$ be the subset of intervals $[b_i,d_i)$ such that $b_i=(h_v(p) , \ord)$ for some $p\in \Crit(h_v^A, (k,+1))$, or $b_i=(h_v(p) , \rel)$ for some $p\in \Crit(h_v^A, (n-k-1,-1))$. Then 
$$\XPH_k(A,h_v)=\bigoplus_{[b_i,d_i)\in J_A^k} \mathcal{I}_{[b_i,d_i)}.$$
\end{theorem}

We can more readily describe the essential classes in dimensions $0$ and $n-1$ in terms of the minimum and maximum values on the different connected components of the boundary. Observe that a compact connected $(n-1)$-dimensional manifold $Y$ embedded in $\R^n$  separates $\R^n$ into two connected open sets, one of which is `inside' $Y$ and one of which is `outside' (this is the unbounded component of the two). This theorem is known as the Jordan-Brouwer separation theorem.  
We  use this to define the connected components of $X=\partial A$ as interior or exterior boundary components.

\begin{definition}
Let $A\subset \R^n$ be a compact $n$-dimensional manifold with boundary $\partial A=X$. Let $\hat{X}$ be a connected component of $X$, and $\hat{A}$ the connected component of $A$ that contains $\hat{X}$. 
We say that $\hat{X}$ is an \emph{interior boundary component} if $\hat{A}\backslash \hat{X}$ is contained in the unbounded connected component of $\R^n\backslash \hat{X}$. We say that $\hat{X}$ is an \emph{exterior boundary component} if $\hat{A}\backslash \hat{X}$ is contained in the bounded connected component  of $\R^n\backslash \hat{X}$.
\end{definition}

\begin{proposition}\label{prop:ess}
Let $A\subset \R^n$ be an $n$-manifold with boundary $X=\partial A$. Let $v$ be a direction such that $h_v^A:A\to \R$ is a Morse function.
Let $\{X_1, \ldots X_k\}$ be the interior boundary components of $X$ and $\{Y_1, \ldots Y_l\}$ be the exterior boundary components of $X$. Then 
$$\Ess_0(A, h_v)=\sum_{j=1}^l \mathcal{I}_{[(\min\{h_v(Y_j)\}, \ord),(\max\{h_v(Y_j)\}, \rel))}$$
and 
$$\Ess_{n-1}(A, h_v)=\sum_{i=1}^k \mathcal{I}_{[(\max\{h_v(X_i)\}, \ord),(\min\{h_v(X_i)\}, \rel))}.$$
\end{proposition}
\begin{proof}
If $A$ is the disjoint union of $A_1, \ldots A_l$ then $\XPH(A)=\oplus_{i=1}^l \XPH(A_i)$. This means that it is sufficient to prove the case where $A$ is connected. We assume $A$ is connected for the remainder of the proof.

Observe that for $M$ a connected $(n-1)$-dimensional manifold we have $\beta_0(M)=1= \beta_{n-1}(M)$ so there is exactly one essential persistent homology interval module in each of these homology dimensions for extended persistent homology of $M$ with respect to $h_v$. 

The interval in $\Ess_0(M,h_v)$ is born at the first appearance of $M$, that is at $(\min\{h_v(M)\}, \ord)$. Since $M$ is connected we have this homology class is trivial in $H_0(M, L)$ for any non-empty subset $L\subset M$. This implies that the death of this interval in $\Ess_0(M, h_v)$ is at parameter $(\max\{h_v(M)\}, \rel)$. We have shown that 
$$\Ess_0(M,h_v)=\mathcal{I}_{[(\min\{h_v(M)\}, \ord),(\max\{h_v(M)\}, \rel))}.$$
Using the symmetry of extended persistent homology for manifolds (Proposition \ref{prop:symmetry}) we have 
$$\Ess_{n-1}(M,h_v)=\mathcal{I}_{[(\max\{h_v(M)\}, \ord),(\min\{h_v(M)\}, \rel))}.$$

Since $X$ is the disjoint union of the interior boundary components $\{X_1, \ldots X_k\}$ and the exterior boundary component $Y$ we have
\begin{align*}
\Ess_0(X,h_v)=&\big(\oplus_{i=1}^n \mathcal{I}_{[(\min\{h_v(X_i)\}, \ord),(\max\{h_v(X_i)\}, \rel))}\big)\\
&\oplus \mathcal{I}_{[(\min\{h_v(Y)\}, \ord),(\max\{h_v(Y)\}, \rel))}
%\big(\oplus_{j=1}^l \mathcal{I}_{[(\min\{h_v(Y_j)\}, \ord),(\max\{h_v(Y_j)\}, \rel))}\big)
\end{align*}
and 
\begin{align*}
\Ess_{n-1}(X,h_v)=&\big(\oplus_{i=1}^n \mathcal{I}_{[(\max\{h_v(X_i)\}, \ord),(\min\{h_v(X_i)\}, \rel))}\big)\\
&\oplus \mathcal{I}_{[(\max\{h_v(Y)\}, \ord),(\min\{h_v(Y_j)\}, \rel))}\big).
\end{align*}

We can use Theorem \ref{thm:submodule} to deduce $\Ess_0(A, h_v)$ and $\Ess_{n-1}(A, h_v)$ from the various persistence modules $\Ess_0(X_i,h_v), \Ess_{n-1}(X_i, h_v), \Ess_0(Y, h_v)$ and $\Ess_{n-1}(Y, h_v)$.

Consider an interior boundary component $X_i$, and let $p_i\in X_i$ be the global minimum of $h_v^{X_i}$. We know that $A$ is contained in the infinite component of $\R^n\backslash X_i$, and $p_i$ must be a $(-)$-critical point for $h_v^A$. This implies that ${[(\min\{h_v(X_i)\}, \ord),(\max\{h_v(X_i)\}, \rel))}$ is not included in $J_A^0$ (where $J_A^k$ is defined in the statement of Theorem \ref{thm:submodule}). Similarly let $\hat{x}_i$ denote the global maximum of $h_v$ over $X_i$. We know that $A$ is contained in the infinite component of $\R^n\backslash X_i$, and $\hat{p}_i$ must be a $(+)$-critical point for $h_v^A$. This implies that  
$${[(\max\{h_v(X_i)\}, \ord),(\min\{h_v(X_i)\}, \rel))}\in J_A^{n-1}.$$  

If we instead consider the exterior boundary component $Y$ then the global minimum of $h_v^{Y}$ will be a $(+)$-critical point for $h_v^A$ and the global maximum of $h_v^{Y}$ will be a $(-)$-critical point for $h_v^A$. This implies that 
$${[(\min\{h_v(Y)\}, \ord),(\max\{h_v(Y)\}, \rel))}\in J_A^{0}$$
but we do not include $ {[(\max\{h_v(Y)\}, \ord),(\min\{h_v(Y)\}, \ord))}$ in $J_A^{n-1}$.

\end{proof}

\section{The Extended Persistent Homology Transform}
\label{sec:xpht} 
\subsection{Background}

The persistent homology transform (PHT) maps the space of shapes embedded in Euclidean space into a space of topological summaries. Instead of comparing the original shapes we can compare their topological transforms. The philosophy is that the persistent homology of a height function in some direction $v$ records geometric information from the perspective of direction $v$. As $v$ changes, the persistent homology classes track geometric features in $M$. The key insight behind the persistent homology transform (PHT) is that by considering the persistent homology from every direction, we preserve all information about the shape.

Before giving the formal definition we should first identify the subsets of space which are allowable shapes, that is the domain of the PHT. We will want our subsets to be reasonably nice. The most general setting for which theoretical properties about the PHT are proved are compact o-minimal sets, which are called \emph{constructible} in \cite{curry2018many}. For the purposes of this paper it is sufficient to know that compact and semi-algebraic or piecewise linear are sufficient conditions for a subset of Euclidean space to be constructible. We will denote the space of constructible subsets of $\R^n$ by $\CS(\R^n)$.

Given an constructible set $M\subset \R^n$, and $v\in S^{n-1}$, let $h_v$ be the corresponding to a height function in direction $v$,
\begin{align*}
h_v&:M \to \R \\
h_v&:x\mapsto \langle x, v\rangle.
\end{align*}
where $\langle\cdot,\cdot\rangle$ denotes the inner product.
We can construct a persistence module $\PH_k(M,h_v)$ by filtering $M$ by the sub-level sets of $h_v$ and taking $k$-dimensional homology groups. The underlying parameter set for the persistence module is $\R$, the attached vector space at $t\in \R$ is $H_k(h_v^{-1}(-\infty,t])$, and for $s\leq t$ the transition map $\varphi_s^t$ is the induced map on homology from the inclusion $h_v^{-1}(-\infty,s] \subseteq h_v^{-1}(-\infty,t]$.

Let $\PM(\mathbb{R})$ denote the standard space of persistence modules over parameter space $\mathbb{R}$.

\begin{definition}
The \emph{Persistent Homology Transform} $\PHT$ of a constructible set $M\in\CS(\R^n)$ is the map $\PHT(M): S^{n-1} \to \PM(\mathbb{R})^n$ that sends a direction to the set of persistence modules by filtering $M$ in the direction of $v$:
\[
	\PHT(M): v \mapsto \left(\PH_0(h_v^M),\PH_1(h_v^M), \ldots, \PH_{n-1}(h_v^M)\right)
\]
where $h_v:M\to \R$, $h_v(x)=\langle x, v\rangle$ is the height function on $M$ in direction $v$.
\end{definition}

Various properties of the PHT have been proved in \cite{turner2014persistent,ghrist2018persistent,curry2018many}. Stability results bound the distance betwen $h_v$ and $h_w$ when $v,w\in S^{n-1}$ are close. This implies that  for each $M\in \CS(R^n)$, its persistent homology transform, $\PHT(M)$, is a continuous function over $S^{n-1}$ when we equip $\PM$ with a Wasserstein metric.

Another very important property about the PHT is its injectivity, that is that for $M_1, M_2 \subset \R^n$, if $\PHT(M_1)=\PHT(M_2)$ then $M_1=M_2$ as subsets of $\R^n$. This was originally proved in~\cite{turner2014persistent} for piecewise linear compact subsets in $\R^2$ and $\R^3$, and then the more general proof was given in~\cite{curry2018many} and independently in~\cite{ghrist2018persistent}.

We can now define a distance between $M_1, M_2$ constructible sets via their persistent homology transforms. We basically just integrate the Wasserstein distances over all the possible directions.

\begin{definition}
Fix $p\in [1,\infty)$ and ambient dimension $n$. Define the distance function $d^{PHT}_p: \CS(\R^n)\times \CS(\R^n) \to \R$ by
$$d^{PHT}_p(M_1, M_2)^p= \int_{v\in S^{n-1}} \sum_{k=0}^{n-1}W_p(\PH_k(M_1, h_v), \PH_k(M_2, h_v))^p \,dv.$$
\end{definition}

\subsection{Now with extended persistence}

We can define a new distance function over $\CS(\R^n)$ by replacing the normal persistent homology with extended persistent homology.
We can construct a definition of a distance between extended persistent homology transforms by replacing the Wasserstein distance between the original persistence modules with those between extended persistence modules.

\begin{definition}
Fix $p\in [1,\infty)$ and ambient dimension $n$. Define the distance function $d^{XPHT}_p: \CS(\R^n)\times \CS(\R^n) \to \R$ by
$$d^{XPHT}_p(M_1, M_2)^p= \int_{v\in S^{n-1}} \sum_{k=0}^{n-1}W_p(\XPH_k(M_1, h_v), \XPH_k(M_2, h_v))^p \,dv.$$
\end{definition}

For the $\PHT$ one theoretical result was the continuity of the $\PHT(M)$ as a function from $S^{n-1}$. This continuity justified the approximation of the PHT by a finite subset of directions. The proofs for the continuity of the PHT can be easily modified to show continuity of the $\XPHT$. Let $\mathscr{E}$ denote the space of extended persistence modules. Then for all $M\in \CS(\R^n)$, the function $\XPHT_k(M): S^{n-1} \to \mathscr{E}$ is continuous when we equip $\mathscr{E}$ with the $p$-Wasserstein distance (for $p\in [1,\infty)$), or the bottleneck distance.

In \cite{skraba2020wasserstein} a stability result for the PHT was proven in the case where $M_1$ and $M_2$ were different embeddings of the same simplicial complex. This bounded the distance between $\PHT(M_1)$ and $\PHT(M_2)$ in terms of the distances between the sets of vertices in the embedding. The proof of this stability theorem can be easily modified to prove an analogous statement for the extended persistent homology transform.

Since the extended persistence module for a filtration by a height function contains strictly more information than the regular persistence module for that height function, the injectivity results for the $\PHT$ will automatically also hold for the $\XPHT$.

\section{Application to binary images} 
\label{sec:Images}

In this section we describe how to interpret a binary digital image as a PL-manifold with boundary, construct boundary curves as PL 1-manifolds, and adapt the results of Section~\ref{sec:morse} to this setting using a simulation of simplicity methodology. 

\subsection{Boundary curves} 
A binary digital image is a two-dimensional array, $P$, with elements called \emph{pixels} taking values in $\{0,1\}$. 
The array is indexed by integers $1 \leq i \leq m$ and $1 \leq j \leq n$, so that $P(i,j)$ is the element in the $i$th row and $j$th column of $P$.   
We can also treat pixels as points in the plane by mapping the array index to a Cartesian coordinate (the first axis is oriented down the page and second from left to right).
Those pixels taking the value `$1$' are defined to be the \emph{foreground} $F := P^{-1}[1]$ and those with value `$0$' are the \emph{background} $G := P^{-1}[0]$.
A small patch of such a binary image array is illustrated in Fig.~\ref{fig:pixelconnectivity}.

 \begin{figure}[h]
    \centering
    \includegraphics[width=0.6 \textwidth]{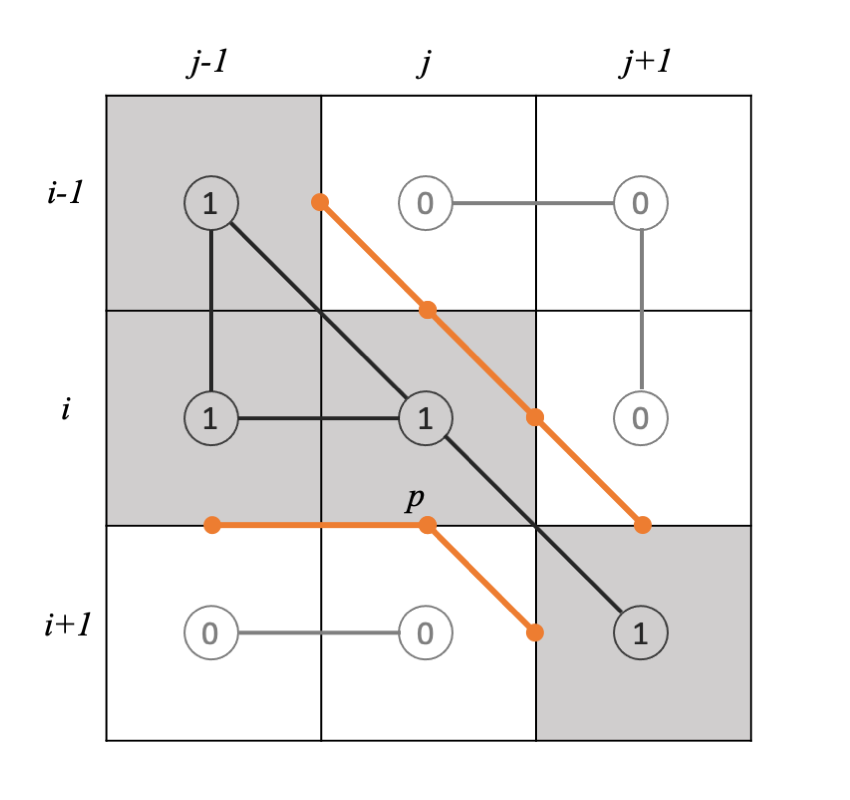} 
\caption{ The rows and columns of a binary digital image are indexed by $i$ and $j$ respectively.  Foreground pixels are labelled `$1$' and connected when 8-adjacent.  Background pixels are labelled `$0$' and connected when 4-adjacent.  Segments of the boundary curves are drawn in orange and the boundary point labelled `$p$' has coordinates $(i+\tfrac{1}{2}, j)$. }
    \label{fig:pixelconnectivity}
\end{figure}

To answer questions about the connectivity of objects represented by the image, we must define a neighbourhood or adjacency relation for each pixel.  Two standard options called 4- and 8-connectivity in digital topology are defined as follows.
 
\begin{definition}
A pixel $(k,l)$ is said to be a \emph{4-adjacent} or \emph{direct} neighbour of $(i,j)$ if their $\ell_1$ distance is exactly 1: 
$\left\vert i-k \right\vert + \left\vert j - k \right\vert = 1$. 
Pixels are \emph{8-adjacent} neighbours if the $\ell_{\infty}$ distance is 1:  $\max \{ \left\vert i-k \right\vert , \left\vert j - k \right\vert \} = 1$.
The 4-neighbourhood of pixel $(i,j)$ consists of its four 4-adjacent neighbours and the 8-neighbourhood is defined similarly.  
 \end{definition}

The connectivity of a set of pixels is then determined according to a specified adjacency relation.  If we choose to use the 8-neighbourhood for pixels in both the foreground and the background, however, counter-intuitive situations may arise such as a simple closed digital curve that does \emph{not} separate the plane into two pieces.  
The resolution of this within digital topology is to treat pixels in the foreground as connected with respect to the 8-neighbourhood and pixels in the background with the 4-neighbourhood, or vice-versa~\cite{kong_digital_1989}. 

We now proceed to construct a set, $\mathcal{C}$, of piecewise-linear curves that subdivide the plane so that each connected component of $\R^2 \setminus \mathcal{C}$ contains pixels of only one type (either foreground or background), and 
such that the digital connected components of $F$ and $G$ are in one-to-one correspondence with those of $\R^2 \setminus \mathcal{C}$. 
As described above, we use 8-connectivity for the foreground and 4-connectivity for the background. 
We assume (and in practice add) a layer of background pixels to any given rectangular array, $P$, to ensure there is a single connected background component surrounding all other components.

\begin{definition}
\emph{Boundary points}.  For every pair of 4-adjacent pixels such that $P(i,j) = 1$ and  $P(k,l) = 0$, define the boundary point $p = ( \tfrac{1}{2}(i+k),  \tfrac{1}{2}(j+l) )$.  
\end{definition} 
There are only four possible configurations.  For example if $(i,j) \in F$ and its direct neighbour $(i+1,j) \in G$, 
then $p = (i +\tfrac{1}{2}, j)$;  the other three cases are simple adjustments to this pattern.  Note that since $(i,j)$ and $(k,l)$ are 4-adjacent, the boundary point has only one coordinate with the $\tfrac{1}{2}$ offset and one remaining an integer.  
See Fig.~\ref{fig:pixelconnectivity} for an illustrative example. 

The next step is to connect pairs of boundary points by line segments in such a way that the foreground and background pixel connectivities are respected.  This is achieved by exhaustive enumeration of $2 \times 2$ pixel patches as illustrated in Fig.~\ref{fig:boundaryedges}.   

\begin{figure}[h]
    \centering
    \includegraphics[width=0.9 \textwidth]{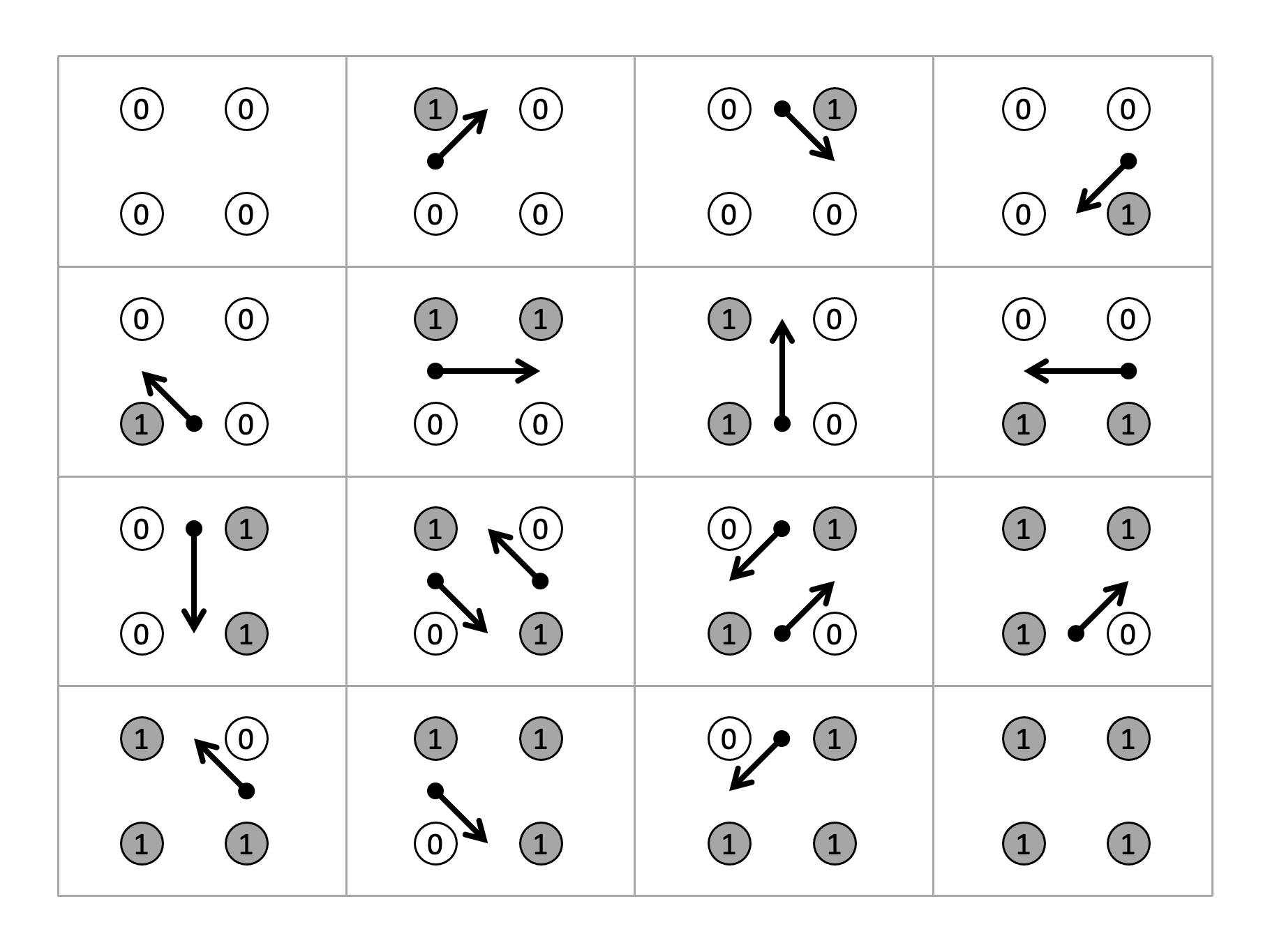} 
\caption{ Each of the $2^4$ possible $2 \times 2$ binary-valued pixel patches showing the associated oriented boundary edges for the case that foreground pixels connect when 8-adjacent. The edge orientation always has the foreground on the left. }
    \label{fig:boundaryedges}
\end{figure}

\begin{lemma}\label{lem:closedcurves}
Let $\mathcal{C} \subset \R^2$ be the union of boundary points and edges derived from a binary digital array $P$. 
The set $\mathcal{C}$ is a disjoint union of simple closed piecewise linear curves. 
\end{lemma}

\begin{proof}
Let $P$ be the $n_r \times n_c$ array with rows indexed by $i=1,\ldots, n_r$ and columns by $j=1,\ldots, n_c$.
By assumption, the outermost rows and columns of $P$ are background, i.e.,  $P(1,j) = P(n_r, j) = P(i, 1) = P(i, n_c) = 0$. 
Each boundary point sits half-way between two 4-adjacent pixels with distinct values, so every boundary point has first coordinate $1 < p_1 < n_r$ and second coordinate $ 1 < p_2 < n_c$. 
It follows that every boundary point must belong to exactly two adjacent $2\times 2$ pixel patches and that  
 every boundary point connects to exactly two boundary edges, by construction.   
As a combinatorial object then, each component of $\mathcal{C}$ is a discrete closed 1-manifold.   
Also by construction (see Fig.~\ref{fig:boundaryedges}) any two boundary edges can only intersect at their endpoints and we conclude that each component of $\mathcal{C}$ is a simple closed PL-curve.
\end{proof}

\begin{lemma}
Let $A$ be the union of components of $\R^2 \setminus \mathcal{C} $ that contain at least one foreground pixel of the binary image array $P$.  Then $A$ is a bounded manifold with boundary $\partial A = \mathcal{C}  $. 
\end{lemma}

\begin{proof}
$A$ is bounded because the image array is finite.  
Each connected component $C_a \in \mathcal{C}$ is a simple closed PL-curve, so $\R^2 \setminus C_a$ consists of two open domains. 
Each connected component of $\R^2 \setminus \mathcal{C}$ is formed by the intersection of a finite number of these domains so is also open and it follows that $A$ is open. 
Clearly $\partial A \subset   \mathcal{C}  $, we must now show that $\partial A \supset  \mathcal{C}  $.  
Let $p \in   \mathcal{C}  $, i.e., $p$ is an arbitrary point on one of the boundary edge segments. 
We can write the coordinates of $p$ as $(i+\epsilon, j+\eta)$ for integers $i,j$ and fractional parts $\epsilon, \eta \in \left[0, 1\right)$.  
We know that each boundary edge divides the $2\times 2$ patch with corners $(i,j)$, $(i+1,j)$, $(i+1,j+1)$, $(i,j+1)$, into two pieces such that at least one of these corners has $P(k,l) = 1$ and this implies that $p \in \partial A$. 
\end{proof} 
 
The above results show that $\cl(A)$ and $X = \partial A$ satisfy the conditions for the theorem(s) of Section~\ref{subsec:XPH_boundary} as $X$ is a finite union of disjoint piecewise-linear 1-manifolds. We then define $B$ to be the closed complement of $A$ in the rectangular domain of the image, $B =  ([1,n_r] \times [1,n_c] )\setminus A $.
A straightforward argument by contradiction shows that no background pixel lies in $A$, so we have $P^{-1}[0] \subset B$.

\begin{remark}
Given a three-dimensional binary array of voxels, $V(i,j,k)$, there are analogous definitions of direct-adjacency between elements, and results that require foreground and background to be viewed with complementary adjacencies to maintain topological consistency~\cite{kong_digital_1989}.   There are also established methods to construct a triangular mesh surface that separates the connected components of foreground and background.  These are termed  `marching cubes algorithms' ~\cite{newman2006survey}.  
\end{remark} 
 
\subsection{ Breaking ties and other practical considerations}

In this section we derive additional results required to extend theorems from Section~\ref{sec:morse} so that they hold for the digital boundary curves.
In particular, Theorem~\ref{thm:submodule} specified that the height function in direction $v$ is a Morse function, i.e., that the critical points are isolated and the critical values are distinct. 
Both these conditions are challenged by the geometry of a digital grid as the boundary curve points lie at integer and half-integer coordinates, and the boundary curve edges are either horizontal, vertical or in one of two diagonal directions.   
Additionally, the direction vectors $v$ are typically chosen to be equal-spaced rational fractions of $\pi$, and will often be perpendicular to some boundary edges. 
This means that when computing the XPHT for equiangular directions $v$ we  expect many vertices of the boundary curves to have the same height with respect to any given $v$. 

Our computations of persistent homology involve height filtrations of boundary curves considered as simplicial complexes.  
The algorithm for computing persistent homology of simplicial complexes orders simplices by their maximal value with lower-dimensional simplices added before higher-dimensional ones if their maximal values are the same. 
It is well understood that the persistent homology of this discrete filtration of complexes gives the same persistence diagram as that of a continuous filtration of a PL-embedding of the complex. 
We do, however, need to explore how a filtration with multiple simplices taking the same height with respect to direction $v$ relates to the critical points of a piecewise-linear Morse function constructed from an arbitrarily close direction $v_t$.

We first need to generalise the notion of $0$-critical point to allow for line segments 
\begin{definition}
Let $\gamma$ be a piecewise-linear non-intersecting curve in $\R^2$ with $m$ vertices traversed in cyclic order, $x_0, x_1,\dots ,x_m=x_0$.  Note that in the following text the indices are assumed to be given as  integers modulo $m$. 
We say $x_i$ is an \emph{isolated $0$-critical vertex} if $h_v(x_{i-1}) > h_v(x_{i})$ and $h_v(x_i)<h_v(x_{i+1})$.  
We say that the line segment from $x_j$ to $x_k$ is a \emph{$0$-critical segment} if $h_v(x_i)=h_v(x_j)$ for all $i=j, j+1, \ldots , k$ and that $h_v(x_{j-1})>h_v(x_j)$ and $h_v(x_{k})<h_v(x_{k+1})$. Denote this line segment as $e(x_j, x_k)$. 
\end{definition}

Observe that if $e=e(x_j, x_k)$ is a 0-critical segment for $h_v$ then the vector $x_k-x_j$ must be perpendicular to $v$, and $h_v$ is constant over $e$.  
Recall that $0$-critical points on the boundary correspond to local minima, and the $0$-critical points which are $(+)$-critical will be local minima as points in $A$. To go from $0$-critical points to $0$-critical segments we need to relax this notion of minima to have non-strict inequalities.

\begin{definition}
We say that a vertex $x_i$ lying on a $0$-critical segment $e$ is \emph{$(+)$-critical} for $h_v^A$ if there exists an $\epsilon>0$ such that for all $a\in B(x_i, \epsilon)\cap A$ we have $h_v(x_i)\leq h_v(a)$. Given the definition of manifold with boundary, if any vertex on a $0$-critical segment $e$ is $(+)$-critical then every vertex on it will be, and we say that the $0$-critical segment is \emph{$(+)$-critical}.
\end{definition}

We now distinguish which of the $0$-critical points and segments on a piecewise linear boundary curve are $(+)$-critical.  
We will use the fact that the orientation of planar triangles is defined by the sign of 
the determinant of a matrix formed from edge vectors as follows. 
First let $DET(x,y)$ be the determinant of a $2\times 2$ matrix with columns $x$ and $y$. 
Given a triangle $\Delta(a,b,c)$ with positive area, the vertices $a,b,c$ are in an anticlockwise order if $DET(c-b, a-b)>0$ and in a clockwise order if $DET(c-b, a-b)<0$. 

The following two geometric lemmas cover the cases where one or both of the edges adjacent to a local minimum is perpendicular to the direction $v$.  

\begin{lemma}\label{lem:min1}
Let $A\subset \R^2$ be a bounded subset whose boundary is the disjoint union of piecewise linear closed curves. Let $\gamma=(x_0,x_1, x_2, \ldots x_m=x_0)$ be a piecewise linear boundary curve of $A$ with vertices listed anticlockwise with respect to $A$.  Fix $v\in S^2$. 
If $x_i$ is an isolated $0$-critical vertex of $h_v^\gamma$, or an endpoint of a $0$-critical segment $e$, then $x_i$ is $(+)$-critical for $h_v^A$ if and only if $DET(x_{i+1}-x_i, x_{i-1}-x_i)>0$.
\end{lemma}

\begin{proof}
There is some $\epsilon>0$ such that the interior of the triangle bounded by $x_i, x_i+ \epsilon(x_{i-1}-x_i)$ and $x_i +\epsilon (x_{i+1}-x_i)$ is either entirely contained in $A$ or is entirely contained in the complement of $A$. For the sake of computations let $y_{i-1}=x_i+ \epsilon(x_{i-1}-x_i)$ and $y_{i+1}=x_i+ \epsilon(x_{i+1}-x_i)$. By assumption we have at least one of $h_v(y_{i+1})>h_v(x_i)$ or $h_v(y_{i-1})>h_v(x_i)$ which implies that the convex hull of $x_i, y_{i-1}$ and $y_{i+1}$ has positive area and  $DET(y_{i+1}-x_i, y_{i-1}-x_i)\neq 0$.

Suppose that $x_i$ is $(+)$-critical for $h_v^A$ which implies that $\Delta(y_{i-1},x_i,y_{i+1})$ is a subset of $A$. Since $\gamma$ traces a boundary curve that is going anticlockwise around $A$ we must have vertices $y_{i-1},x_i,y_{i+1}$ in an anticlockwise order. This implies $DET(y_{i+1}-x_i, y_{i-1}-x_i)>0$.

If $x_i$ is not $(+)$-critical then the opposite holds: we have $\Delta(y_{i-1},x_i,y_{i+1})$ is not contained in $A$, that $y_{i-1},x_i,y_{i+1}$ are in a clockwise order and thus $DET(y_{i+1}-x_i, y_{i-1}-x_i)<0$.
\end{proof}

If the point $x_i$ is contained strictly inside a $0$-critical segment then we need an alternative approach. This will also be useful when the points on the curve are close to co-linear, because we want to avoid any possible issues with floating point errors in computations. 

\begin{lemma}\label{lem:min2}
Let $A\subset \R^2$ be a bounded subset whose boundary is the disjoint union of piecewise linear closed curves. 
Let $\gamma=(x_0,x_1, x_2, \ldots x_m=x_0)$ be a piecewise linear boundary curve of $A$ with vertices listed anticlockwise with respect to $A$.  Fix $v\in S^2$. 
Let $x_i$ an isolated $0$-critical vertex or a vertex in a $0$-critical segment $e$ with respect to the function $h_v^\gamma$. 
Furthermore suppose that $\angle(x_{i-1}, x_i, x_{i+1}) > \pi/2 $. 
Let $w_i$ denote the rotation of the vector $x_{i+1}-x_i$ anticlockwise by $\pi/2$. Then $x_i$ is $(+)$-critical for $h_v^A$ if and only if $w_i\cdot v > 0$.
\end{lemma}
\begin{proof}

If $w_i\cdot v=0$ then we have $h_v(x_i)$ lying strictly between $h_v(x_{i-1})$ and $h_v(x_{i+1})$ contradicting the assumption $x_i$ is $0$-critical, so we know $w_i\cdot v\neq 0$.

Since $\gamma$ is traced anticlockwise around $A$ and $\angle(x_{i-1}, x_i, x_{i+1}) > \pi/2 $ we know that $w$, the rotation of $x_{i+1}-x_i$ anticlockwise by $\pi/2$, will point into $A$ from $x_i$. Set $y=x_i+w_i$. If we rotate anticlockwise around $x_i$ from direction $w_i$ we encounter $x_{i-1}-x_i$ within a rotation of $\pi$. This follows from $\angle (y, x_i, x_{i+1})=\pi/2$ and  $\angle(x_{i-1}, x_i, x_{i+1}) > \pi/2 $. Since $w_i$ points into $A$ from $x_i$, for small $\epsilon>0$, we can cover $A\cap B(x_i, \epsilon)$ by triangles $\Delta(x_{i-1}, x_i, y)$ and $\Delta(y, x_i, x_{i+1})$. 

If $w_i\cdot v > 0$ then $h_v(y)>h_v(x_i)$. Every point $a\in \Delta(x_{i-1}, x_i, y)$ can be written as an affine combination $a=a_1x_{i-1}+a_2x_{i} +a_3y$. For this $a$,
$$h_v(a)= a_1h_v(x_{i-1})+a_2h_v(x_{i}) +a_3h_v(y)\geq h_v(x_i)$$ as $h_v(x_{i-1}), h_v(y)\geq h_v(x_i)$. Similarly every point $a\in \Delta(y,x_i, x_{i+1})$ also satisfies $h_v(a)\geq h_v(x_i)$. Together these imply that $x_i$ is $(+)$-critical.

If $w_i\cdot v <0$ then $h_v$ decreases along $t \mapsto x_i + tw_i$, showing that for all $\epsilon>0$ there are points in $a_\epsilon\in A \cap B(x_i, \epsilon)$ with $h_v(a_\epsilon)<h_v(x_i)$. this implies $x_i$ is not $(+)$-critical.
\end{proof}

We are now ready to state a related theorem to Theorem \ref{thm:submodule} for PL subsets of the plane where we drop the Morse condition.

\begin{theorem}\label{thm:imagesubmodule}
Let $A\subset \R^2$ be a $2$-dimensional piecewise linear manifold with boundary $X=\partial A$. Fix $v\in S^1$.
The $0$-dimensional persistent homology of $h_v^X:X\to \R$ can be written as
$$\PH_0(X,h_v^X)=\oplus_{i=1}^m \mathcal{I}_{[h_v(y_{j_i}),d_i)}$$
where $y_{j_1}, \ldots y_{j_m}$ are the set of vertex representatives, and $d_1, \ldots d_m \in \R\cup \infty$. Here we have only included intervals with positive length.

Let $J^{\ord}$ be the subset of $\{1, 2, \ldots m\}$ such that $d_i$ is finite and $y_{j_i}$ is $(+)$-critical for $h_v^A$. Then
$$\Ord_0(A,h_v^A)=\oplus_{i\in J^{\ord}} \mathcal{I}_{[(h_v(y_{j_i}), \ord),(d_i, \ord))}.$$

Now let $J^{\rel}$ be the subset of $\{1, 2, \ldots m\}$ such that  $d_i$ is finite but $y_i$ is not $(+)$-critical for $h_v^A$. 
$$\Rel_1(A, h_v^A)= \oplus_{i\in J^{\rel}} \mathcal{I}_{[(d_i, \rel), (h_v(y_{j_i}),\rel))}.$$
\end{theorem}

\begin{proof}
If $h_v^A$ is a Morse function the result follows directly from Theorem \ref{thm:submodule}, so suppose that $h_v^A$ is not a Morse function. Recall that since $A\subset \R^2$ is a $2$-dimensional piecewise linear manifold with boundary $X=\partial A$, a sufficient condition for $h_v^A$ to be Morse will be that all the vertices in $A$ have distinct values under $h_v^A$.

Let $v_t$ be the rotation of $v$ anticlockwise by $t$. Given $v$ there is an $\epsilon>0$ such that for all $t<\epsilon$ we have $h_v(x)<h_v(y)$ implies $h_{v_{t}}(x)<h_{v_t}(y)$. 
We can now break the ties that imply $h_v^A$ is not Morse; where $h_v(x)=h_v(y)$ we have $h_{v_t}(x)\neq h_{v_t}(y)$. We choose $\epsilon>0$ small enough that $h_{v_t}^A$ is a Morse function for all $t<\epsilon$.

A vertex $y_{j_i}$ will be an isolated $0$-critical vertex for $h_v^X$ if and only if it is an isolated vertex for $h_{v_t}^X$, as the order of the heights of $y_{j_i-1},  y_{j_i}$ and $y_{j_i+1}$ are the same under both $h_v$ and $h_{v_t}$. Since $DET(y_{j_i+1}-y_{j_i}, y_{j_i-1}-y_{j_i})>0$ is independent of $v$ we know that whether or not $y_{j_i}$ is $(+)$-critical is the same under $h_v^A$ and $h_{v_t}^A$ by Lemma \ref{lem:min1}. For  ease of reference later in the proof, set $x_{j_i}=y_{j_i}$.  

Now suppose $e(x_k, x_l)$ is a $0$-critical segment for $h_v^A$ with $y_{j_i}\in e(x_k, x_l)$. Since $h_{v_t}^A$ is Morse, all the vertices in $e(x_k, x_l)$ take distinct values for $h_{v_t}^X$, with exactly one of  $x_k$ or $x_l$ now an isolated $0$-critical point.   
Denote this endpoint by $x_{j_i}$. 
Since we choose $v_t$ to be a small anticlockwise rotation of $v$, this choice will be a consistent tie-break for all $0<t<\epsilon$. 
Again since $DET(x_{j_i+1}-x_{j_i}, x_{j_i-1}-x_{j_i})>0$ is independent of $v$ we know that whether or not $x_{j_i}$ is $(+)$-critical is the same under $h_v^A$ and $h_{v_t}^A$ by Lemma \ref{lem:min1}.

By construction,  $h_v(x_{j_i})=h_v(y_{j_i})$ for all $i$ so we have $$\PH_0(X,h_v^X)=\oplus_{i=1}^m \mathcal{I}_{[h_v(x_{j_i}),d_i)}.$$
The remainder of the proof is an argument in continuity.  For $t\in (0, \epsilon)$ we have 
$$\PH_0(X,h_{v_t}^X)=\oplus _{i=1}^m \mathcal{I}_{[h_v(x_{j_i}),d_{i,t})}$$ for some $d_{i,t}\in \R$. Since $\lim_{t\to 0^+}v_t=v$ we have $\lim_{t\to 0^+}\PH_0(X,h_{v_t}^X)=\PH_0(X,h_v^X)$ and thus 
$\lim_{t\to 0^+} d_{i,t}=d_i$ for all $i$.

Since each $x_{j_i}$ is $(+)$-critical with respect to $h_v^A$ if and only if it is $(+)$-critical with respect to $h_{v_t}^A$ we can apply Theorem \ref{thm:submodule} to say for all $t\in (0,\epsilon)$ that
$$\Ord_0(A,h_{v_t}^A)=\oplus_{i\in J^{\ord}} \mathcal{I}_{[(h_{v_t}(y_{j_i}), \ord),(d_{i,t}, \ord))}$$
and 
$$\Rel_1(A, h_{v_t}^A)= \oplus_{i\in J^{\rel}} \mathcal{I}_{[(d_{i,t}, \rel), (h_{v_t}(y_{j_i}),\rel))}.$$
Taking the limit as $t\to 0^+$ completes the proof.
\end{proof}

\section{Implementation details}
\label{sec:Implementation}

Using the theory developed in the previous sections we have implemented a package in R which takes as input a binary image and outputs the extended persistent homology transform of the foreground of that image. 
The R-package is available at \href{https://github.com/james-e-morgan/xpht}{https://github.com/james-e-morgan/xpht}.  
The paragraphs below describe a simple example to illustrate the sequence of steps followed when using the package. 
We finish this section with a fun application using the XPHT to cluster the shapes of letters from various standard fonts.  

Let $A$ denote the foreground of the binary image and $X$ the boundary between the foreground and background as constructed in the previous section. 
The user chooses the number of directions $K$, and the unit vectors are set to $v_i=(\cos(2\pi i/K), \sin(2\pi i)/K)$.
We can compute the extended persistent homology of $A$ for directions $v$ and $-v$ from the regular persistent homology of $X$ in direction $v$ together with knowledge of the minimum and maximum values of $h^X_v$ on each boundary curve.
Therefore, when the number of directions is even, the computational time for the XPHT is halved.  
If the user has a collection of shapes that require centring~\cite{turner2014persistent} then $K$ is required to be a multiple of four.

The first step is to label the connected components of the foreground and construct the oriented boundary curves around each of the components, labelling which curves are interior and which are exterior. 
Note that by Lemma \ref{lem:closedcurves} the boundary between the foreground and background is a disconnected collection of closed curves. 
This set of boundary curves is independent of the choice of direction and is computed only once. 
For an example see Figure \ref{fig:binary}. 

\begin{SCfigure}[20]
         \includegraphics[width=0.35\textwidth]{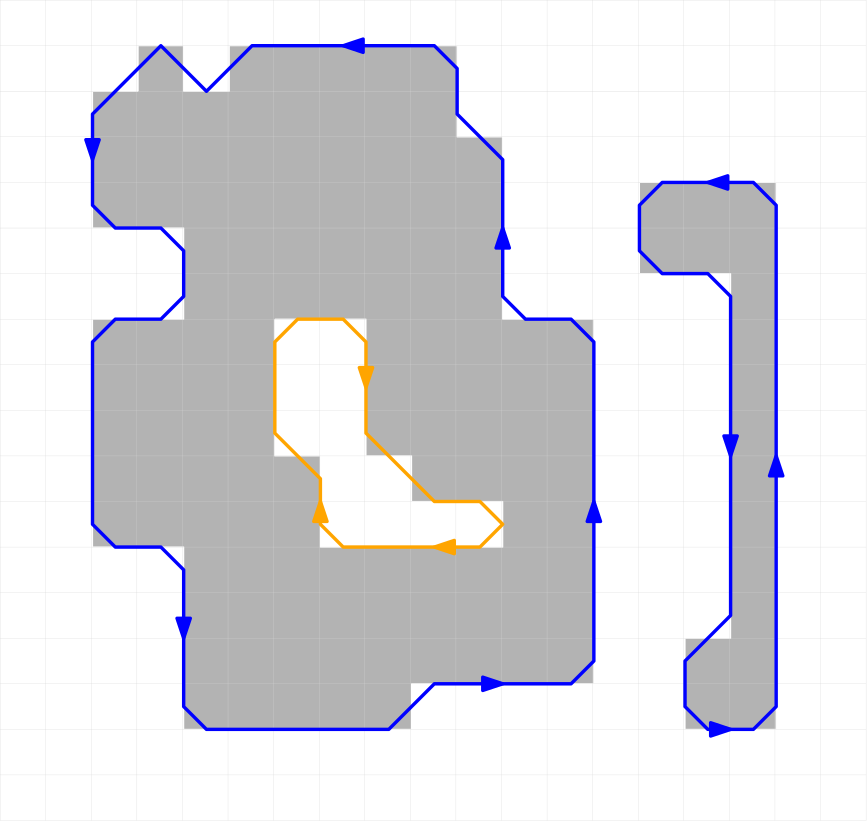}
     \caption{The input binary image $A$ with foreground in grey. The boundary curves $\partial A$ are oriented anticlockwise with the interior curve in orange and the exterior curves in blue. These curves are constructed using the rules illustrated in Figure~\ref{fig:boundaryedges}.}
\label{fig:binary}
\end{SCfigure}

For each direction $v$ the regular 0-dimensional persistent homology of the boundary curves can be computed very efficiently using the union-find data structure.%~\cite{}. 
Our implementation also identifies a $0$-critical vertex that represents the birth of a component in the filtration of $\partial A$; see Figure~\ref{fig:PH0}. 

\begin{SCfigure}[20]
	 \includegraphics[width=0.35\textwidth]{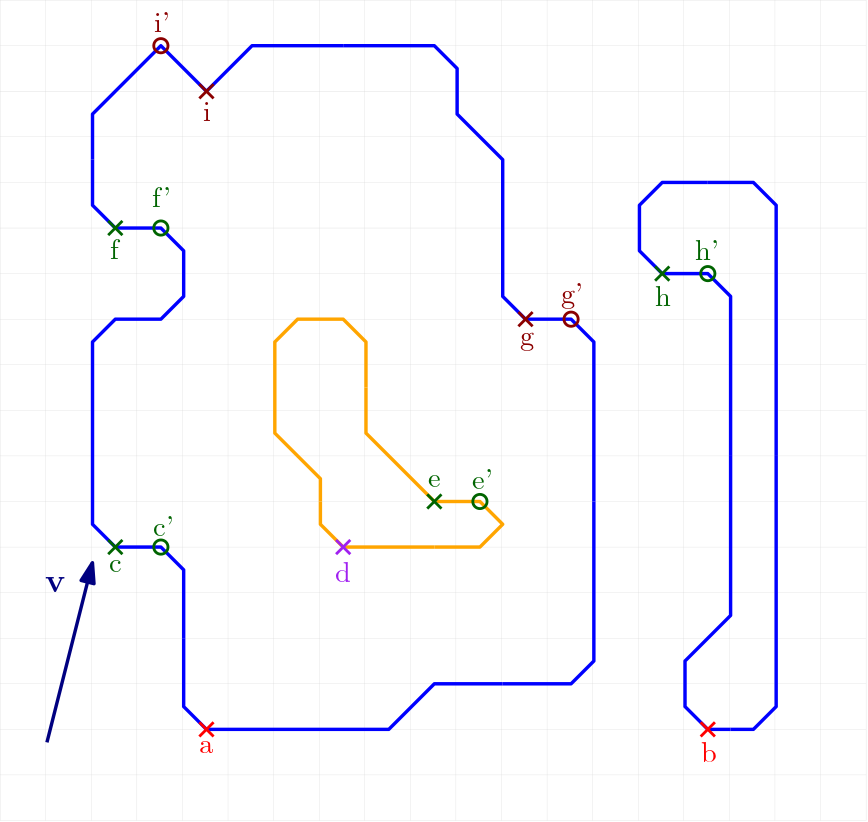}
         \caption{The critical points for $PH_0(\partial A,v)$ for the given direction $v$. The vertices marked with crosses are $0$-critical points and correspond to birth events in $PH_0(\partial A,v)$. Vertices marked with circles are $1$-critical and  cause a death in $PH_0(\partial A,v)$.  
The same letter label is given to the paired birth and death events of a persistent homology class from  $PH_0(\partial A,v)$.}
 \label{fig:PH0}
\end{SCfigure}

Using Lemma \ref{lem:min1} or \ref{lem:min2} it is also quick to determine which $0$-critical points are positive critical or negative critical for the foreground.  We thus can label all of the ordinary persistent homology classes as either $(+)$-critical or $(-)$-critical.  
This is illustrated for the example in Figure \ref{fig:localmin}.

\begin{SCfigure}[20]
\centering
         \includegraphics[width=0.35\textwidth]{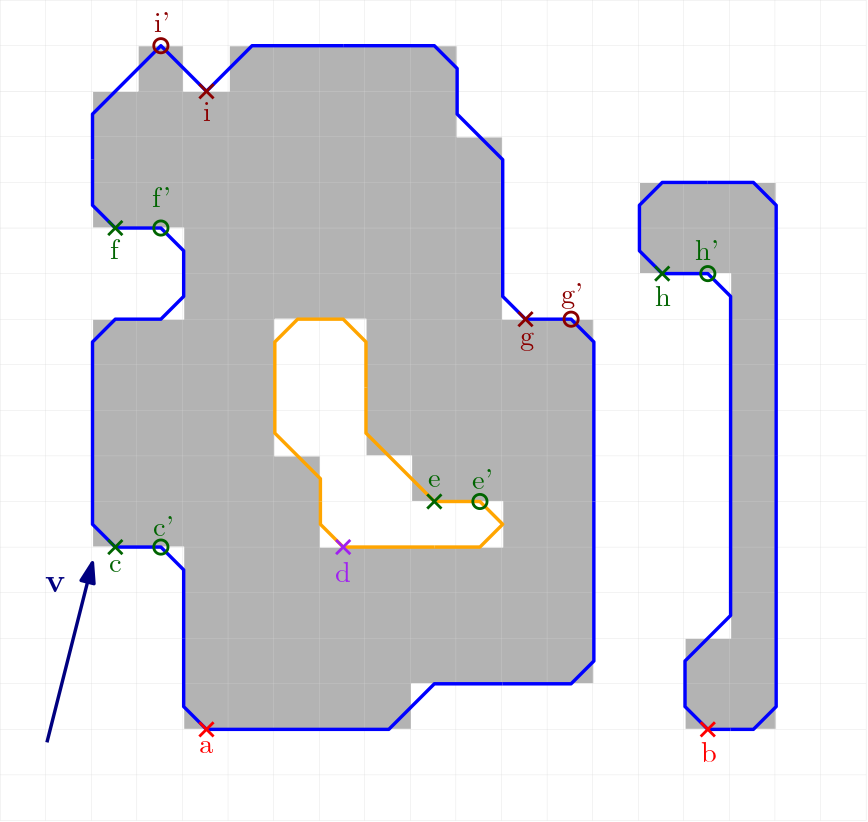}
         \caption{Identifying the boundary curve local minima that are also local minima of $h_v$ on the foreground. The $0$-critical points for $X$ that correspond to births of finite lifetime persistence classes in $PH_0(X,v)$ are $c,e,f,g,h$ and $i$. We have of these $c, e,f$ and $h$ are local minima for $A$ and thus $(+)$-critical points. The remaining ($g$ and $i$) are $(-)$-critical.}
         \label{fig:localmin}
\end{SCfigure}

Using Theorem \ref{thm:imagesubmodule} we can compute the ordinary and relative persistent homology for $h_v^A$ from the persistent homology of $h_v^X$ together with information about which $0$-critical isolated vertices and $0$-critical segments are $(+)$-critical. 
Applying the duality result from Corollary~\ref{cor:duality} we  deduce the ordinary and relative persistence modules for direction $-v$ from those for direction $v$. In our worked example:
\begin{align*}
\Ord_0(A,v)=&\mathcal{I}_{[(h_v(c),\ord), (h_v(c'),\ord))}\oplus \mathcal{I}_{[(h_v(e), \ord)(h_v(e'), \ord))}\\
& \oplus  \mathcal{I}_{[(h_v(f), \ord)(h_v(f'), \ord))} \oplus  \mathcal{I}_{[(h_v(h), \ord)(h_v(h'), \ord))} \\
\Rel_1(A,v)=&\mathcal{I}_{[(h_v(i'),\rel), (h_v(i),\rel))}\oplus \mathcal{I}_{[(h_v(g'), \rel)(h_v(g), \rel))}\\
\Ord_0(A,(-v))=&\mathcal{I}_{[(-h_v(i'),\ord), (-h_v(i),\ord))}\oplus \mathcal{I}_{[(-h_v(g'), \ord)(-h_v(g), \ord))} \\
\Rel_1(A,(-v))=&\mathcal{I}_{[(-h_v(c),\rel), (-h_v(c'),\rel))}\oplus \mathcal{I}_{[(-h_v(e), \rel)(-h_v(e'), \rel))}\\
& \oplus  \mathcal{I}_{[(-h_v(f), \rel)(-h_v(f'), \rel))} \oplus  \mathcal{I}_{[(-h_v(h), \rel)(-h_v(h'), \rel))} .
\end{align*}

To compute the essential classes we  use Proposition \ref{prop:ess}.   Each of the boundary curves is labelled as interior or exterior.   We compute the essential classes by finding the minimum and maximum values of $h_v^X$ on these boundary curves. This is illustrated for our running example in Figure ~\ref{fig:ess}.

\begin{SCfigure}[20]
	 \includegraphics[width=0.35\textwidth]{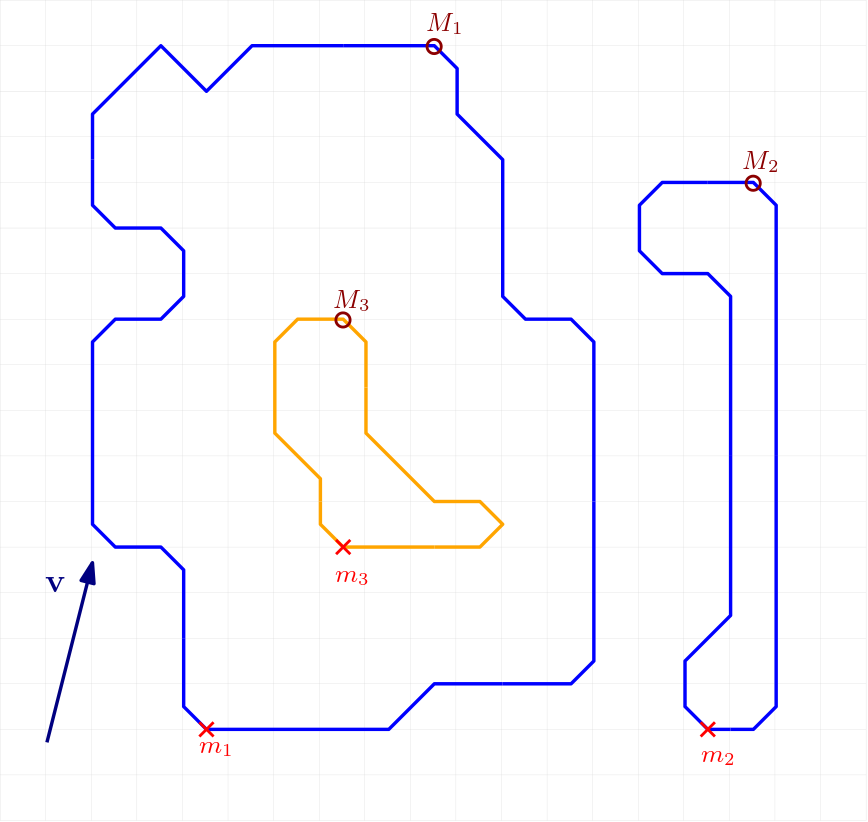}
         \caption{The minima and maxima points of $h_v^X$ for each boundary curve in $X$. The exterior curves have (minimum,maximum) pairs labelled $(m_1, M_1)$ and $(m_2, M_2)$ and the interior curve has the pair of $(m_3, M_3)$.}
         \label{fig:ess}
\end{SCfigure}

Using the notation of the figure, the essential classes for the foreground $A$ are therefore 
$$\Ess_0(A,v)=\mathcal{I}_{[(h_v(m_1),\ord),(h_v(M_1),\rel)}\oplus \mathcal{I}_{[(h_v(m_2),\ord),(h_v(M_2),\rel)}$$ and 
$$\Ess_1(A,v)=\mathcal{I}_{[(h_v(M_3), \ord), (h_v(m_3), \rel)}.$$ 
We then infer the essential persistence modules for direction $-v$ to be 
 $$\Ess_0(A,-v)=\mathcal{I}_{[(-h_v(M_1),\ord),(-h_v(m_1),\rel)}\oplus \mathcal{I}_{[(-h_v(M_2),\ord),(-h_v(m_2),\rel)}$$ and $$\Ess_1(A,-v)=\mathcal{I}_{[(-h_v(m_3), \ord), (-h_v(M_3), \rel)}.$$

\bigskip 

\begin{example}
We now briefly describe results from an XPHT analysis of the capital letter $A$ and the lower-case letter $g$ rendered using over 90 standard fonts. 
Each letter was created as a small binary image ($130 \times 130$) using an 84pt font size; these are shown in Figures~\ref{fig:fontA} and~\ref{fig:fontg}. 
The XPHT for each letter was computed using $K= 32$ directions.  
Fonts vary in their letter placement with respect to a baseline, so we centred the XPHT summary for each shape using the method outlined in \cite{turner2014persistent}. 
We did not scale the data as the letters have the same specified font size; this allows the different heights and widths to serve as characteristics of the font. 
When comparing the XPHT summaries of two shapes, we also did not need to consider angular alignment as the images are generated with a consistent orientation.  

We computed all pairwise distances between the XPHT summaries using both the $1$- and $2$-Wasserstein metrics.  
To demonstrate the types of shape features the XPHT captures, we use multi-dimensional scaling (MDS) to assign planar coordinates to each image.  
The plots in Figures~\ref{fig:fontA} and~\ref{fig:fontg} show that the XPHT distances capture the difference between serif and sans-serif versions of the letter $A$, and between single- and double-storey versions of the letter $g$. 
Of particular note is the font `Chalkduster' (label 32) which has a textured look with small holes and rough boundary; the XPHT distances don't make this an outlier for the letter $A$s. 
Chalkduster $g$ is an outlier for that set because the bowl doesn't create a closed 1-cycle. 
It's also worth noting that the XPHT distances vary nicely from the double-storey letter `g's which have $\beta_1 = 2$, to the single-storey `$g$'s with $\beta_1 = 1$, and that fonts such as those labelled 62 and 24, which look double-storey but have $\beta_1 = 1$ are placed in the middle of the MDS plot. 

These letters are included in the R-package release and more details about the analysis are provided in the vignettes. 
\end{example}

\begin{figure}[h]
\hspace{-2cm}         \includegraphics[height=6.5cm]{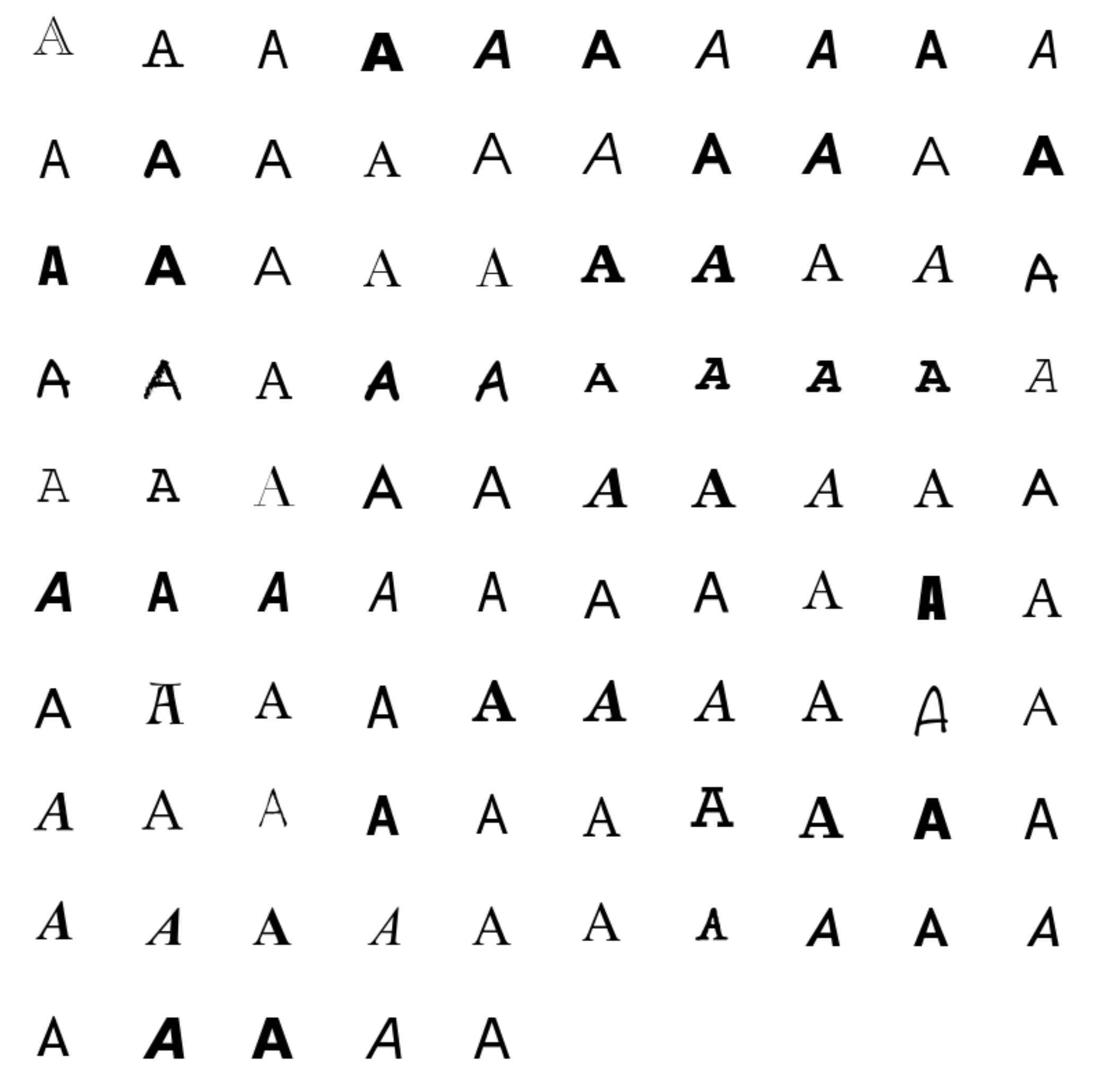}       
          \includegraphics[height=6.5cm]{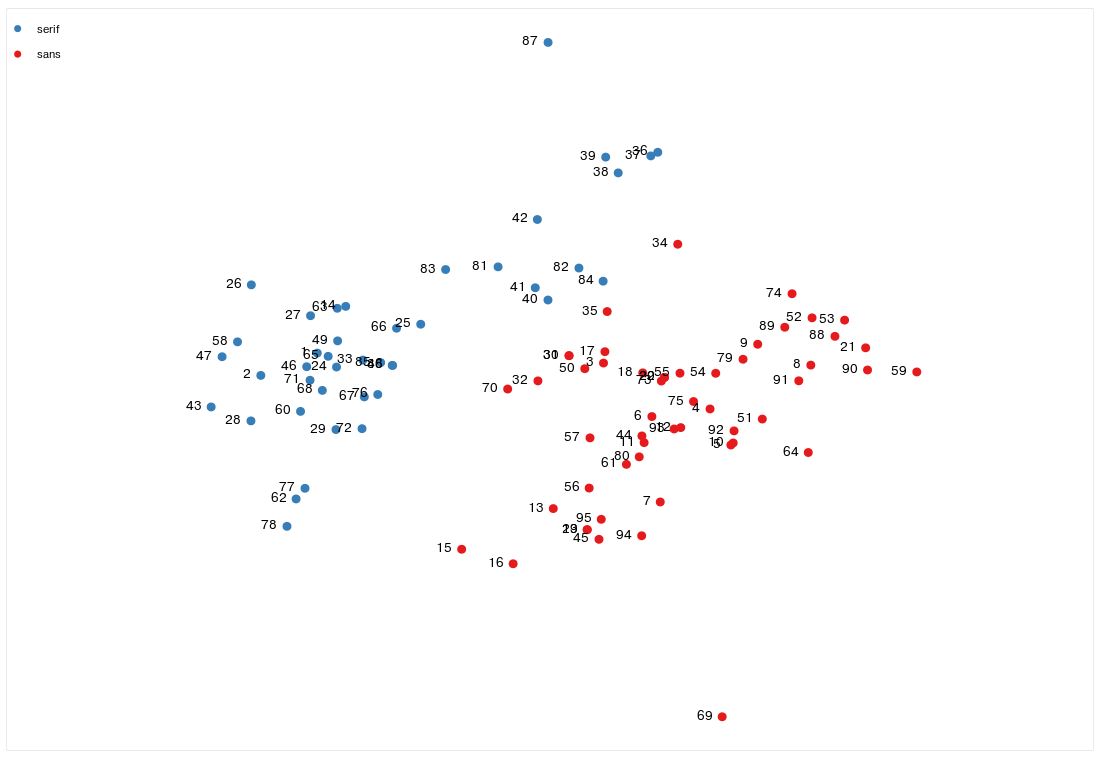}
         \caption{Upper case `A' rendered in a variety of fonts.  The letter shapes are numbered 1--95 reading left to right, top to bottom in the 10 by 10 grid.   The $2$-Wasserstein distances between each pair of letters are visualised using MDS in two dimensions. This shows a separation between serif `A' and sans serif `\textsf{A}'  fonts.  The upper outlier labelled 87 is `Trattatello' and has the smallest height and counter in this set. The lower outlier labelled 69 is `Noteworthy-Light', a large round simple script font. On the far right is letter 59 (`Impact'), notable for having a narrow body width but heavy weight. }
\label{fig:fontA}  
\end{figure}

\begin{figure}[h]
\hspace{-2cm}           \includegraphics[height=6.5cm]{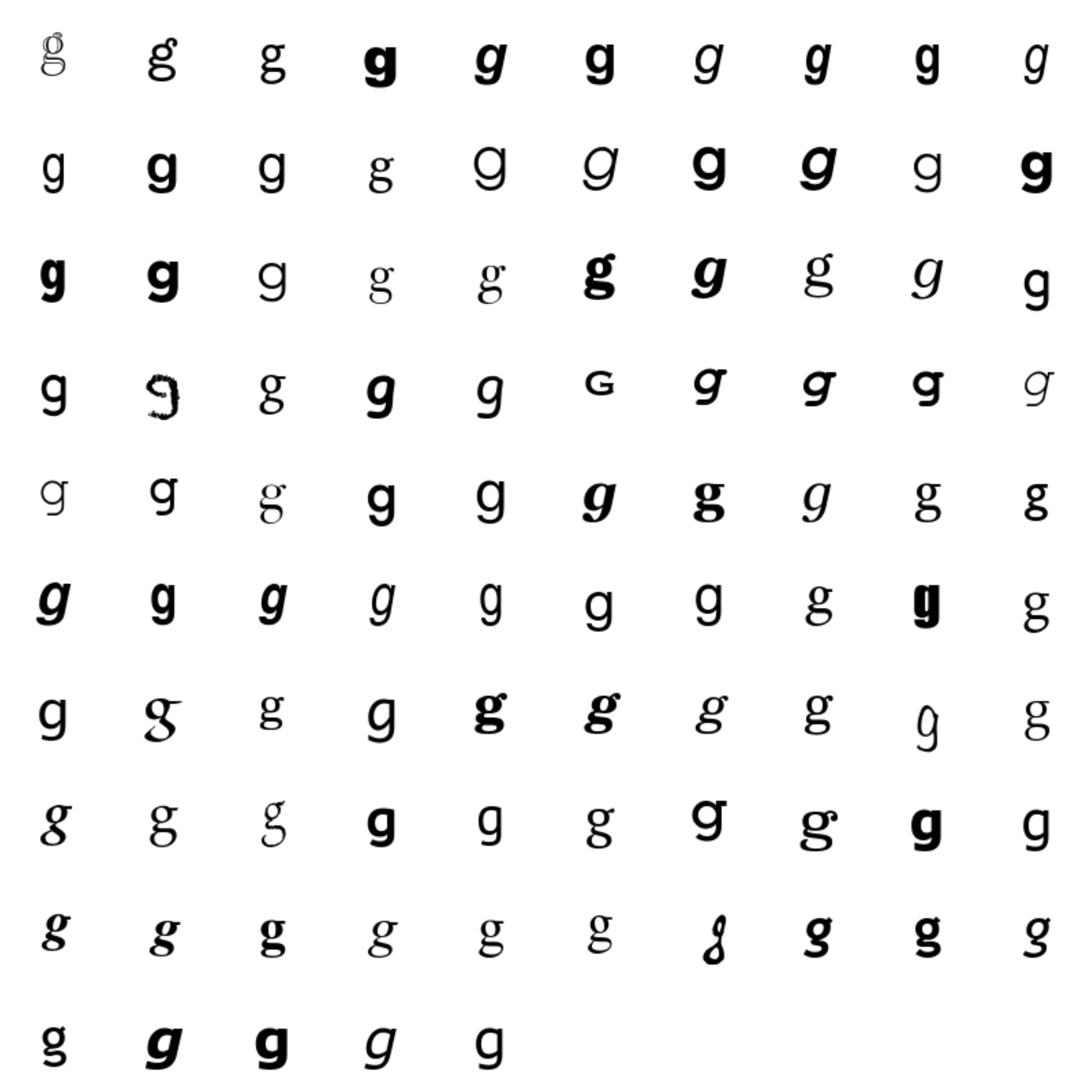}       
          \includegraphics[height=6.5cm]{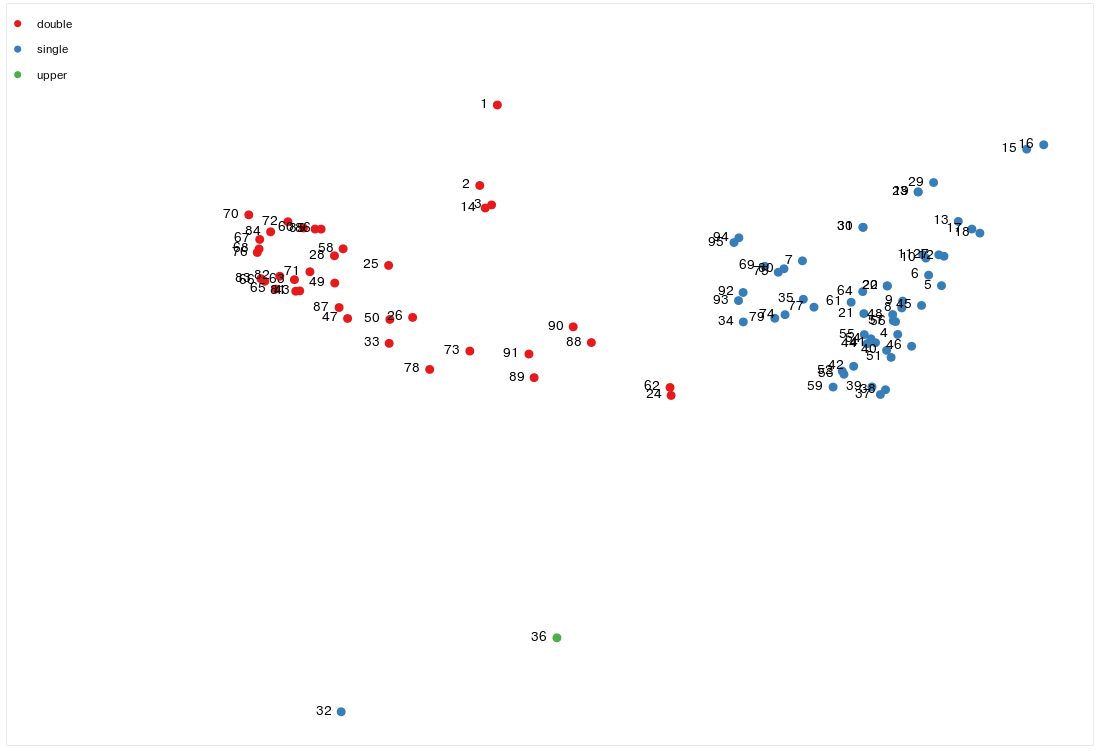}
         \caption{Lower case `g' rendered in the same fonts and same order as used for `A'.  The $2$-Wasserstein distances between each pair of letters are again visualised using MDS in $2$ dimensions.  In this case, there is good separation between single storey `$g$' and double storey `g' font shapes. The lower outlier labelled 32 is `Chalkduster' and has no essential 1-cycles; its nearest neighbour is 36 `Copperplate' which renders in upper-case form.  The first sample font is `Academy Engraved'; it is placed up above the others because it has additional essential 1-cycles due to its outlined stroke style.  On the far right are letters 15 and 16 (forms of `Avantgarde'); these have the roundest bowls and largest counters, i.e., large circular upper holes.  } 
     \label{fig:fontg}
\end{figure}

\section{Future directions}

This paper presents a new approach to computing persistent homology for manifolds with boundary by exploiting relationships between the extended persistent homology of a manifold with boundary to that of just the boundary. 
Although the focus here has been on height functions of embedded shapes in Euclidean space it is reasonable to expect that similar results could hold for other kinds of functions, such as radial functions. 
Future directions of research also include considering generalisations to stratified spaces, adapting ideas from stratified Morse theory as developed by Goresky and MacPherson~\cite{goresky1988stratified}.

Other areas to explore are theoretical properties of the XPHT.  In particular, intuitively we would expect better stability results than for the PHT as we can introduce new essential classes with small support without dramatically changing the extended persistent homology transform.

\bibliographystyle{plain}
\bibliography{Bibliography}

\begin{thebibliography}{10}

\bibitem{amezquita2022measuring}
Erik~J Am{\'e}zquita, Michelle~Y Quigley, Tim Ophelders, Jacob~B Landis, Daniel
  Koenig, Elizabeth Munch, and Daniel~H Chitwood.
\newblock Measuring hidden phenotype: Quantifying the shape of barley seeds
  using the euler characteristic transform.
\newblock {\em in silico Plants}, 4(1):diab033, 2022.

\bibitem{bauer2020universality}
Ulrich Bauer, Magnus~Bakke Botnan, and Benedikt Fluhr.
\newblock Universality of the bottleneck distance for extended persistence
  diagrams.
\newblock {\em arXiv preprint arXiv:2007.01834}, 2020.

\bibitem{bauer2014induced}
Ulrich Bauer and Michael Lesnick.
\newblock Induced matchings of barcodes and the algebraic stability of
  persistence.
\newblock In {\em Proceedings of the thirtieth annual symposium on
  Computational geometry}, pages 355--364, 2014.

\bibitem{braess1974morse}
Dietrich Braess.
\newblock Morse-theorie f{\"u}r berandete mannigfaltigkeiten.
\newblock {\em Mathematische Annalen}, 208(2):133--148, 1974.

\bibitem{carlsson2019parametrized}
Gunnar Carlsson, Vin De~Silva, Sara Kali{\v{s}}nik, and Dmitriy Morozov.
\newblock Parametrized homology via zigzag persistence.
\newblock {\em Algebraic \& Geometric Topology}, 19(2):657--700, 2019.

\bibitem{chazal2014observable}
Fr{\'e}d{\'e}ric Chazal, William Crawley-Boevey, and Vin De~Silva.
\newblock The observable structure of persistence modules.
\newblock {\em arXiv preprint arXiv:1405.5644}, 2014.

\bibitem{cohen2009extending}
David Cohen-Steiner, Herbert Edelsbrunner, and John Harer.
\newblock Extending persistence using poincar{\'e} and lefschetz duality.
\newblock {\em Foundations of Computational Mathematics}, 9(1):79--103, 2009.

\bibitem{crawford2020predicting}
Lorin Crawford, Anthea Monod, Andrew~X Chen, Sayan Mukherjee, and Ra{\'u}l
  Rabad{\'a}n.
\newblock Predicting clinical outcomes in glioblastoma: an application of
  topological and functional data analysis.
\newblock {\em Journal of the American Statistical Association},
  115(531):1139--1150, 2020.

\bibitem{crawley2015decomposition}
William Crawley-Boevey.
\newblock Decomposition of pointwise finite-dimensional persistence modules.
\newblock {\em Journal of Algebra and its Applications}, 14(05):1550066, 2015.

\bibitem{curry2018many}
Justin Curry, Sayan Mukherjee, and Katharine Turner.
\newblock How many directions determine a shape and other sufficiency results
  for two topological transforms.
\newblock {\em arXiv preprint arXiv:1805.09782}, 2018.

\bibitem{de2011dualities}
Vin De~Silva, Dmitriy Morozov, and Mikael Vejdemo-Johansson.
\newblock Dualities in persistent (co) homology.
\newblock {\em Inverse Problems}, 27(12):124003, 2011.

\bibitem{edelsbrunner_alexander_2012}
Herbert Edelsbrunner and Michael Kerber.
\newblock Alexander {{Duality}} for {{Functions}}: {{The Persistent Behavior}}
  of {{Land}} and {{Water}} and {{Shore}}.
\newblock In {\em Proceedings of the {{Twenty-eighth Annual Symposium}} on
  {{Computational Geometry}}}, {{SoCG}} '12, pages 249--258, {New York, NY,
  USA}, 2012. {ACM}.

\bibitem{ghrist2018persistent}
Robert Ghrist, Rachel Levanger, and Huy Mai.
\newblock Persistent homology and euler integral transforms.
\newblock {\em Journal of Applied and Computational Topology}, 2(1):55--60,
  2018.

\bibitem{goresky1988stratified}
Mark Goresky and Robert MacPherson.
\newblock Stratified {M}orse theory.
\newblock In {\em Stratified {M}orse {T}heory}, pages 3--22. Springer, 1988.

\bibitem{grunert2019pl}
Romain Grunert, Wolfgang K{\"u}hnel, and G{\"u}nter Rote.
\newblock Pl morse theory in low dimensions.
\newblock {\em arXiv preprint arXiv:1912.05054}, 2019.

\bibitem{jankowski1972functions}
A~Jankowski and E~Rubinsztejn.
\newblock Functions with non-degenerate critical points on manifolds with
  boundary.
\newblock {\em Commentationes Mathematicae}, 16(1), 1972.

\bibitem{kong_digital_1989}
T~Yung Kong and Azriel Rosenfeld.
\newblock Digital topology: Introduction and survey.
\newblock {\em Computer Vision, Graphics, and Image Processing},
  48(3):357--393, 1989.

\bibitem{Milnor}
J.~Milnor.
\newblock {\em Morse theory}.
\newblock Based on lecture notes by M. Spivak and R. Wells. Annals of
  Mathematics Studies, No. 51. Princeton University Press, Princeton, N.J.,
  1963.

\bibitem{newman2006survey}
Timothy~S Newman and Hong Yi.
\newblock A survey of the marching cubes algorithm.
\newblock {\em Computers \& Graphics}, 30(5):854--879, 2006.

\bibitem{shboul2019feature}
Zeina~A Shboul, Mahbubul Alam, Lasitha Vidyaratne, Linmin Pei, Mohamed~I
  Elbakary, and Khan~M Iftekharuddin.
\newblock Feature-guided deep radiomics for glioblastoma patient survival
  prediction.
\newblock {\em Frontiers in Neuroscience}, page 966, 2019.

\bibitem{skraba2020wasserstein}
Primoz Skraba and Katharine Turner.
\newblock Wasserstein stability for persistence diagrams.
\newblock {\em arXiv preprint arXiv:2006.16824}, 2020.

\bibitem{tang2022topological}
Wai~Shing Tang, Gabriel~Monteiro da~Silva, Henry Kirveslahti, Erin Skeens, Bibo
  Feng, Timothy Sudijono, Kevin~K Yang, Sayan Mukherjee, Brenda Rubenstein, and
  Lorin Crawford.
\newblock A topological data analytic approach for discovering biophysical
  signatures in protein dynamics.
\newblock {\em PLoS computational biology}, 18(5):e1010045, 2022.

\bibitem{turner2014persistent}
Katharine Turner, Sayan Mukherjee, and Doug~M Boyer.
\newblock Persistent homology transform for modeling shapes and surfaces.
\newblock {\em Information and Inference: A Journal of the IMA}, 3(4):310--344,
  2014.

\bibitem{zhang2021mfcis}
Yanping Zhang, Jing Peng, Xiaohui Yuan, Lisi Zhang, Dongzi Zhu, Po~Hong, Jiawei
  Wang, Qingzhong Liu, and Weizhen Liu.
\newblock Mfcis: an automatic leaf-based identification pipeline for plant
  cultivars using deep learning and persistent homology.
\newblock {\em Horticulture research}, 8, 2021.

\end{thebibliography}
  
 \end{document}